\documentclass[11pt]{amsart}
\usepackage[english]{babel}
\usepackage{amsthm} 
\usepackage{amssymb}
\usepackage{amsmath}
\usepackage{hyperref}
\usepackage{amsfonts}
\usepackage{verbatim}
\usepackage{stmaryrd}
\usepackage{fancyhdr}
\usepackage{mathrsfs}
\usepackage{graphicx}
\usepackage{pdfsync}
\newtheorem{theorem}{Theorem}[section]

\newtheorem{lemma}[theorem]{Lemma}
\newtheorem{proposition}[theorem]{Proposition}

\numberwithin{equation}{section}
\newcommand{\N}{\mathbb{N}}
\newcommand{\nn}{\mathbb{N}}
\newcommand{\R}{\mathbb{R}}
\newcommand{\rr}{\mathbb{R}}
\newcommand{\CC}{\mathbb{C}}
\newcommand{\cc}{\mathbb{C}}
\newcommand{\lv}{\left[}
\newcommand{\rv }{\right]}
\newcommand{\eps}{\varepsilon}
\newcommand {\be}{\begin{equation}}
\newcommand {\ee}{\end{equation}}
\newcommand {\ba}{\begin{array}}
\newcommand {\ea}{\end{array}}

\usepackage{varioref}
 
\def\Rom#1{\uppercase\expandafter{\romannumeral #1}}

\def\polhk#1{\setbox0=\hbox{#1}{\ooalign{\hidewidth \lower1.5ex\hbox{`}\hidewidth\crcr\unhbox0}}} 

\begin{document}

\title[Pseudospectra and analyticity sector of a Ornstein-Uhlenbeck semigroup]{
Pseudospectral phenomena preventing the analyticity sector of a Ornstein-Uhlenbeck semigroup to be ruled by the spectral angle of its generator}

\author{Karel Pravda-Starov}

\address{\noindent \textsc{Karel Pravda-Starov, Univ Rennes, CNRS, IRMAR - UMR 6625, F-35000 Rennes, France}}
\email{karel.pravda-starov@univ-rennes.fr}
\address{\noindent \textsc{Karel Pravda-Starov, Abraham de Moivre International Research Laboratory, CNRS - Imperial College London, Imperial College London, South Kensington, London SW7 2AZ, UK}}
\email{k.pravda-starov@imperial.ac.uk}


\begin{abstract}
We study elliptic Ornstein-Uhlenbeck operators in $L^2$ spaces with respect to invariant measures. These operators have discrete spectra and generate analytic semigroups. The angle defining the analyticity sector of a Ornstein-Uhlenbeck semigroup is not always equal to the complementary angle of the spectral angle of its generator. 
In this work, we study the link between these two angles. More specifically, we provide a necessary and sufficient criterion ensuring that they are complementary, and explain how pseudospectral phenomena account for these two angles to fail to be complementary. 
\end{abstract}

\maketitle

\section{Introduction}

\subsection{Ornstein-Uhlenbeck operators}

We consider elliptic Ornstein-Uhlenbeck operators
\begin{equation}\label{jen0}
P=\frac{1}{2}\sum_{i,j=1}^{n}q_{i,j}\partial_{x_i,x_j}^2+\sum_{i,j=1}^nb_{i,j}x_j\partial_{x_i}=\frac{1}{2}\textrm{Tr}(Q\nabla_x^2)+\langle Bx,\nabla_x\rangle, \quad x \in \rr^n,
\end{equation}
where $Q=(q_{i,j})_{1 \leq i,j \leq n}$ and $B=(b_{i,j})_{1 \leq i,j \leq n}$ are real $n \times n$-matrices, with $Q \gg 0$ symmetric positive definite. In the following, we denote $\langle A,B \rangle$ and $|A|^2$ the scalar operators
\begin{equation}\label{not}
\langle A,B\rangle=\sum_{j=1}^nA_jB_j, \quad |A|^2=\langle A,A\rangle=\sum_{j=1}^nA_j^2,
\end{equation}
when $A=(A_1,...,A_n)$ and $B=(B_1,...,B_n)$ are vector-valued operators. Notice that $\langle A,B\rangle\neq\langle B,A\rangle$ in general, since e.g., 
$\langle \nabla_x,Bx\rangle=\langle Bx,\nabla_x\rangle +\textrm{Tr}(B).$ 
The associated Ornstein-Uhlenbeck semigroup $(T(t))_{t \geq 0}$ has the following explicit representation
\begin{equation}\label{gen}
(T(t)f)(x)=\frac{1}{(2\pi)^{\frac{n}{2}}\sqrt{\det Q_t}}\int_{\rr^n}e^{-\frac{1}{2}\langle Q_t^{-1}y,y\rangle}f(e^{tB}x-y)dy, \quad t>0,
\end{equation}
with
\begin{equation}\label{pav0}
Q_t=\int_0^{t}e^{sB}Qe^{sB^T}ds,
\end{equation}
where $B^T$ denotes the transpose matrix of $B$.
This formula is due to Kolmogorov~\cite{kolmo}. On the other hand, the existence of an invariant measure $\mu$ for the semigroup $(T(t))_{t \geq 0}$, that is, a probability measure on $\rr^n$ verifying 
$$\forall t \geq 0, \forall f \in C_b(\rr^n), \quad \int_{\rr^n}(T(t)f)(x)d\mu(x)=\int_{\rr^n}f(x)d\mu(x),$$
where $C_b(\rr^n)$ stands for the space of continuous and bounded functions on $\rr^n$, is known to be equivalent~\cite[Section~11.2.3]{DaPrato}
to the following localization of the spectrum of~$B$,
\begin{equation}\label{pav1}
\sigma(B) \subset \mathbb{C}_-=\{z \in \mathbb{C} : \textrm{Re }z<0\}.
\end{equation}
When this condition holds, the invariant measure is unique and is given by $d\mu(x)=\rho(x)dx$, where the density with respect to the Lebesgue measure is
\begin{equation}\label{pav2}
\rho(x)=\frac{1}{(2\pi)^{\frac{n}{2}}\sqrt{\det Q_{\infty}}}e^{-\frac{1}{2}\langle Q_{\infty}^{-1}x,x\rangle},
\end{equation}
with 
\begin{equation}\label{pav3}
Q_{\infty}=\int_0^{+\infty}e^{sB}Qe^{sB^T}ds.
\end{equation}
It is known from~\cite{fuhrman,Metafune_al2002} that Ornstein-Ulhenbeck semigroups extend to strongly continuous one-parameter semigroups of contractions on $L^p_{\mu}=L^p(\rr^n,d\mu)$ spaces, $1 \leq p<+\infty$, with respect to invariant measures, which are analytic when $1<p<+\infty$, but not for $p=1$.

In the work~\cite{sector}, Chill, Fa\v{s}angov\'a, Metafune and Pallara study the sector of analyticity of elliptic Ornstein-Uhlenbeck semigroups in $L_{\mu}^p$ spaces. 
They prove that the angle of the analyticity sector $0<\theta_p \leq \frac{\pi}{2}$ is given by
$$\textrm{cotan }\theta_p=\frac{\sqrt{(p-2)^2+p^2(\textrm{cotan }\theta_2)^2}}{2\sqrt{p-1}},$$
when $1<p<+\infty$, where the angle $0<\theta_2 \leq \frac{\pi}{2}$ corresponding to the $L^2_{\mu}$ case, is determined by the formula
\begin{equation}\label{rt2}
\textrm{cotan }\theta_2=\|I_n+2Q^{-1/2}Q_{\infty}B^TQ^{-1/2}\|.
\end{equation}
Here $\|\cdot\|$ denotes the spectral norm $\|\cdot\|_2$, that is, the matrix norm induced by the Euclidean norm on $\rr^n$.
The result of~\cite{sector} (Theorem~2, see also Remark~6) shows that $(T(t))_{t \geq 0}$ the Ornstein-Uhlenbeck semigroup on $L_{\mu}^p$, with $1<p<+\infty$, extends to an analytic contraction semigroup on the angular sector
$$\Sigma_{\theta_p}:=\{z \in \cc^* : |\arg z| <\theta_p\} \cup \{0\},$$
that is, $T$ extends as a family of bounded operators $(T(z))_{z \in \Sigma_{\theta_p}}$ such that the mapping $z \in \Sigma_{\theta_p}\setminus \{0\} \mapsto T(z) \in \mathcal{L}(L_{\mu}^p)$ is analytic in $\Sigma_{\theta_p}\setminus \{0\}$ satisfying
$$\forall z_1,z_2 \in \Sigma_{\theta_p}, \quad T(z_1+z_2)=T(z_1)T(z_2),$$
$$\forall 0<\eps<\theta_p,\exists M_{\eps}>0,\forall z \in \Sigma_{\theta_p-\eps},  \quad  \|T(z)\|_{\mathcal{L}(L_{\mu}^p)} \leq M_{\eps}$$
and 
$$\forall 0<\eps<\theta_p, \forall f \in L_{\mu}^p, \quad  \lim_{\substack{z \to 0\\ z \in \Sigma_{\theta_p-\eps}\setminus \{0\}}}T(z)f=f \textrm{ in } L_{\mu}^p.$$ 
Furthermore, it is shown that the angle $\theta_p$ is sharp in the sense that, if $(T(t))_{t \geq 0}$ extends to an analytic semigroup on the angular sector $\Sigma_{\theta}$ for $0<\theta \leq \frac{\pi}{2}$, then $\theta \leq \theta_p$. 
As noticed in~\cite{sector} (Remark~2), the spectrum of the matrix 
$$Q^{-1/2}Q_{\infty}B^TQ^{-1/2},$$ 
determines the analyticity sector of the Ornstein-Uhlenbeck semigroup since the spectral norm of the skew-adjoint matrix 
\begin{equation}\label{skew}
A=I_n+2Q^{-1/2}Q_{\infty}B^TQ^{-1/2},
\end{equation} 
is equal to its spectral radius
$$\|I_n+2Q^{-1/2}Q_{\infty}B^TQ^{-1/2}\|=r(I_n+2Q^{-1/2}Q_{\infty}B^TQ^{-1/2}).$$
Indeed, the skew-adjointness property $A^T=-A$ directly comes from the definitions (\ref{pav0}) and (\ref{pav3}),
$$\forall t \geq 0, \quad Q_{\infty}=Q_t+e^{tB}Q_{\infty}e^{tB^T},$$
which implies that 
\begin{equation}\label{ai4}
0=\frac{d}{dt}(Q_t+e^{tB}Q_{\infty}e^{tB^T})|_{t=0}=Q+BQ_{\infty}+Q_{\infty}B^T=\frac{1}{2}Q^{1/2}(A+A^T)Q^{1/2}.
\end{equation}
On the other hand, the spectrum of the elliptic Ornstein-Uhlenbeck operator (\ref{jen0}) acting on the $L_{\mu}^p$ space, with $1 \leq p<+\infty$, has been explicitly described by Metafune, Pallara and Priola in~\cite{Metafune_al2002}. When $1 < p<+\infty$, the spectrum of the Ornstein-Uhlenbeck operator $(P,D_p)$ considered as the infinitesimal generator of the Ornstein-Uhlenbeck semigroup $(T(t))_{t \geq 0}$ in $L_{\mu}^p$ defined in (\ref{gen}), is discrete and is only composed by eigenvalues. It is independent of the parameter $1<p<+\infty$ and is solely determined by the spectrum of the drift matrix $B$,
\begin{equation}\label{rt4}
\sigma(P)=\Big\{\sum_{\lambda \in \sigma(B)}\lambda k_{\lambda} :\ k_{\lambda} \in \mathbb{N}\Big\} \subset \{z \in \cc : \textrm{Re }z \leq 0\},
\end{equation}
where $\mathbb{N}$ denotes the set of non-negative integers. 
The above localization of the spectrum of $P$ follows from the localization of the spectrum of the drift matrix (\ref{pav1}). 
This result actually holds for Ornstein-Uhlenbeck operators which are merely hypoelliptic~\cite{Metafune_al2002} (Theorem~3.1). In the $L_{\mu}^2$ case, this result may also be recovered from the singular space theory developed for degenerate non-selfadjoint quadratic operators~\cite{HPS,OPPS2}. When $p=1$, the spectrum of $P$ in the $L_{\mu}^1$ space is the closed left half-plane
$$\sigma(P)=\{z \in \cc : \textrm{Re }z \leq 0\}.$$
In this case, each complex number $z$ with a negative real part $\textrm{Re }z<0$ is actually then an eigenvalue~\cite{Metafune_al2002} (Theorem~5.1). According to the characterization recalled below (\ref{dani1}), this accounts for the lack of analyticity of the Ornstein-Uhlenbeck semigroup when $p=1$.

The description of the spectrum (\ref{rt4}) indicates that the spectral angle $\theta_{\textrm{spec}}$ of the Ornstein-Uhlenbeck operator $P$ is equal to the spectral angle of the matrix $B$, 
\begin{equation}\label{dan3}
0 \leq \theta_{\textrm{spec}}:=\textrm{arctan}\Big(\sup_{z \in \sigma(B)}\Big|\frac{\textrm{Im }z}{\textrm{Re }z}\Big|\Big) < \frac{\pi}{2},
\end{equation}
that is, the smallest angle $0 \leq \theta <\frac{\pi}{2}$ satisfying $\sigma(P) \subset -\overline{\Sigma}_{\theta}$ is actually equal to the smallest angle $0 \leq \theta <\frac{\pi}{2}$ satisfying $\sigma(B) \subset -\overline{\Sigma}_{\theta}$.
\begin{figure}[ht]
\caption{Spectrum of the Ornstein-Uhlenbeck operator $P$.}
\centerline{\includegraphics[scale=0.75]{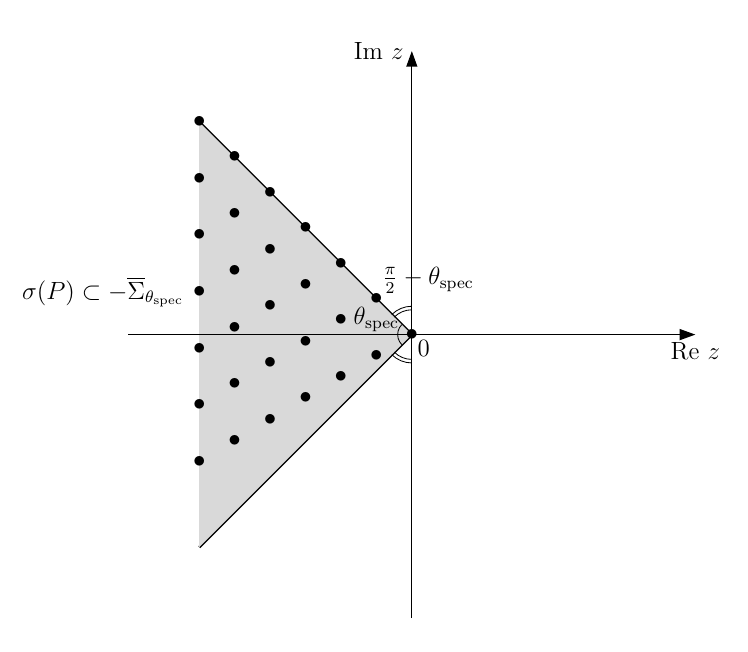}}
\end{figure}
The standard theory of analytic semigroups (see e.g. Chapter~9, paragraph~6 in~\cite{kato} and chapter~2 in~\cite{lunardi}) implies that the angle of the analyticity sector $\theta_p$ satisfies the bound
$$0<\theta_p +\theta_{\textrm{spec}}\leq \frac{\pi}{2}.$$
Indeed, the analyticity of the semigroup $(T(t))_{t \geq 0}$ in the angular sector $\Sigma_{\theta}$, with $0<\theta \leq \frac{\pi}{2}$, is characterized by the following conditions for its infinitesimal generator
\begin{multline}\label{dani1}
\exists \omega \in \rr, \quad \Big\{z \in \cc : z \neq \omega, \ |\textrm{arg}(z-\omega)|<\frac{\pi}{2}+\theta\Big\} \subset \cc \setminus \sigma(P) ; \\ 
\forall 0<\eps<\theta, \exists M_{\eps}>0,  \forall z \in \cc,\ z \neq \omega, \ |\textrm{arg}(z-\omega)|<\frac{\pi}{2}+\theta-\eps, \quad \|(P-zI)^{-1}\| \leq \frac{M_{\eps}}{|z-\omega|}.
\end{multline}  
The authors of the work~\cite{sector} were surprised to notice that there are situations when the angle of the analyticity sector of the Ornstein-Uhlenbeck semigroup is not equal to the complementary angle of the spectral angle of its generator
$$0<\theta_p+\theta_{\textrm{spec}} < \frac{\pi}{2}.$$ 
A two-dimensional example of this situation is given in~\cite{sector} (Example~1) for $Q=I_2$ and 
$$B=\left(\begin{array}{cc}
-\delta& \lambda \\
0 & -\delta      
\end{array}\right),$$
with $\delta>0$, $\lambda \in \rr^*$. In this case, the spectral angle $\theta_{\textrm{spec}}=0$ is zero, since
$\sigma(B)=\{-\delta\} \subset \rr_-$, and we indeed have
$$0<\theta_2+\theta_{\textrm{spec}}=\theta_2=\textrm{arctan}\Big(\frac{2\delta}{|\lambda|}\Big)<\frac{\pi}{2},$$
since 
\begin{equation}\label{dan2}
Q_{\infty}=\int_0^{+\infty}e^{-2\delta s}\left(\begin{array}{cc}
1 & s \lambda \\
0 & 1      
\end{array}\right)\left(\begin{array}{cc}
1 & 0 \\
s \lambda & 1      
\end{array}\right)ds=\frac{1}{4\delta^3}\left(\begin{array}{cc}
2\delta^2+\lambda^2 & \lambda \delta \\
\lambda \delta  & 2\delta^2      
\end{array}\right)
\end{equation}
and
$$Q_{\infty}B^T=-\frac{1}{4\delta}\left(\begin{array}{cc}
    2\delta       & \lambda  \\
    -\lambda       & 2\delta
    \end{array}\right) \ \textrm{ and } \sigma(I_2+2Q_{\infty}B^T)=\Big\{\pm i\frac{\lambda}{2\delta}\Big\},$$
imply that 
$$\textrm{cotan } \theta_2=\|I_2+2Q_{\infty}B^T\|=r(I_2+2Q_{\infty}B^T)=\frac{|\lambda|}{2\delta}.$$

In the present work, we consider specifically the $L_{\mu}^2$-case, and we aim at establishing a necessary and sufficient criterion ensuring when the spectral angle of the Ornstein-Uhlenbeck operator $P$ on $L^2_{\mu}$ is actually equal to the complementary of the angle of the analyticity sector $\theta_2$ of its associated semigroup, i.e., 
$$\theta_2 +\theta_{\textrm{spec}}= \frac{\pi}{2}.$$
We explain in this work that some pseudospectral phenomena account for the fact that the spectral angle of the operator $P$, and the angle of the analyticity sector of the associated semigroup may fail to be complementary
$$0<\theta_2 +\theta_{\textrm{spec}}< \frac{\pi}{2}.$$ 
As seen below, the Ornstein-Uhlenbeck operator
$$P=\frac{1}{2}\textrm{Tr}(Q\nabla_x^2)+\langle Bx,\nabla_x\rangle, \quad x \in \rr^n,$$ 
is selfadjoint on $L^2_{\mu}$ if and only if the condition 
$$\frac{1}{2}Q_{\infty}^{-1}Q+B^T=0,$$
holds. When this condition does not hold, the operator is non-selfadjoint. In this case, it is well-known that the resolvent of such an operator may take very large values in norm far away from its spectrum and that the spectrum may be very unstable under small perturbations. These two features are highly related. Indeed, studying the level lines of the norm of the resolvent of an operator provides substantial information about its spectral stability. We recall from~\cite{roch} that when $\mathcal{A}$ is a closed unbounded linear operator with a dense domain on a complex Hilbert space $H$, its $\varepsilon$-pseudospectrum
$$\sigma_{\varepsilon}(\mathcal{A})=\Big\{z \in \cc:  \ \|(\mathcal{A}-z)^{-1}\| \geq \frac{1}{\varepsilon} \Big\}, \quad \eps>0,$$
with the convention that $\|(\mathcal{A}-z)^{-1}\|=+\infty$ for any point $z \in \sigma(\mathcal{A})$, can be defined in an equivalent way in term of the spectrum of perturbations of the operator
$$\sigma_{\eps}(\mathcal{A})=\bigcup_{\mathcal{B} \in \mathcal{L}(H), \ \|\mathcal{B}\|_{\mathcal{L}(H)} \leq \eps}{\sigma(\mathcal{A}+\mathcal{B})},$$
where $\mathcal{L}(H)$ stands for the set of bounded linear operators on $H$. There exists an extensive literature on the notion of pseudospectrum. We refer the reader to~\cite{trefethen,trefethen2} and the references therein for a detailed account of this topic.

The study of pseudospectrum is non-trivial only for non-selfadjoint operators, or more precisely for non-normal operators. Indeed, the classical formula 
\begin{equation}\label{1}
\forall z \not\in \sigma(\mathcal{A}), \quad \|(\mathcal{A}-z)^{-1}\| = \frac{1}{\textrm{dist}(z,\sigma(\mathcal{A}))}, 
\end{equation}
implies that the resolvent of a normal operator 
cannot blow up far away from its spectrum~\cite{kato} (Chap.~V, Sect.~3.8, formula (3.31)). It ensures the stability of the spectrum under small perturbations 
$$\sigma_{\eps}(\mathcal{A}) = \{z \in \cc :\  \textrm{dist}(z,\sigma(\mathcal{A})) \leq \eps\}.$$
However, the formula (\ref{1}) does not hold anymore for non-normal operators and the resolvent of such operators may become very large in norm far away from the spectrum. It implies that the spectra of these operators may become very unstable under small perturbations. 
We explain in the following how these possible phenomena of blow-up in norm far from the spectrum for the resolvent of non-normal Ornstein-Uhlenbeck operators account for the fact that the spectral angle of the operator and the angle of analyticity sector of the associated semigroup may fail to be complementary
$$0<\theta_2+\theta_{\textrm{spec}} < \frac{\pi}{2}.$$

\subsection{Statement of the main result and some examples}\label{main}
The following statement is the main result contained in this work:

\bigskip

\begin{theorem}\label{th}
Let $P$ be the Ornstein-Ulhenbeck operator defined in \emph{(\ref{jen0})}, where $Q$ and $B$ are real $n \times n$-matrices, with $Q \gg 0$ symmetric positive definite and $B$ satisfying the spectral condition \emph{(\ref{pav1})}. Let $Q_{\infty}$ be the symmetric positive definite matrix defined in \emph{(\ref{pav3})}. Setting
$$A=I_n+2Q^{-1/2}Q_{\infty}B^TQ^{-1/2} \ \textrm{ and }\ D=Q^{1/2}Q_{\infty}^{-1}Q^{1/2},$$
the spectral angle of the operator $P$ in $L^2_{\mu}$ denoted $\theta_{\emph{\textrm{spec}}}$, and the angle of the analyticity sector of its associated semigroup $(T(t))_{t \geq 0}$ denoted $\theta_2$, are complementary, that is,
$$\theta_2 +\theta_{\emph{\textrm{spec}}}= \frac{\pi}{2},$$
if and only if the following finite intersection of kernels is non-zero
\begin{equation}\label{a21}
\Big( \bigcap_{k=0}^{n-1}\emph{\textrm{Ker}}
\big((A^2+\|A\|^2)D^{k}\big)\cap \emph{\textrm{Ker}}
\big([A,D^k]\big)\Big)\cap \rr^{n}\neq\{0\},
\end{equation}
where $\|\cdot\|$ denotes the spectral norm $\|\cdot\|_2$, and $[A,D^k]$ denotes the commutator of the two matrices $A$ and $D^k$.
\end{theorem}

\bigskip

The following examples highlight how Theorem~\ref{th} applies in various cases:

\bigskip

\noindent
\textit{Example 1. (Case of normal Ornstein-Uhlenbeck operators)}.
Let  $Q \gg 0$ be a symmetric positive definite real $n \times n$ matrix and $B \in M_n(\rr)$ satisfying the spectral condition (\ref{pav1}). As seen in Section~\ref{uhl} (see (\ref{ai5}) and (\ref{ai6})), the Ornstein-Uhlenbeck operator (\ref{jen0}) is normal when acting on $\sqrt{\rho}^{-1}\mathscr{S}(\rr^n) \subset L^2_{\mu}$, i.e.,
$$\forall u \in \mathscr{S}(\rr^n), \quad [P,P^*](\sqrt{\rho}^{-1}u)=0,$$ 
with $P^*$ being the $L^2_{\mu}$-adjoint of $P$;
if and only if the two matrices $Q_{\infty}^{-1}Q$ and $B^T$ commutes, i.e., 
$$[Q_{\infty}^{-1}Q,B^T]=0.$$ 
When this condition holds, we notice that the two matrices $A$ and $D$ commutes as well as
\begin{align*}
& \ [A,D]=2[Q^{-1/2}Q_{\infty}B^TQ^{-1/2},Q^{1/2}Q_{\infty}^{-1}Q^{1/2}]\\
=&\  2(Q^{-1/2}Q_{\infty}B^TQ_{\infty}^{-1}Q^{1/2}-Q^{1/2}B^TQ^{-1/2})
=2Q^{-1/2}Q_{\infty}(B^TQ_{\infty}^{-1}Q-Q_{\infty}^{-1}QB^T)Q^{-1/2}\\
= & \ -2Q^{-1/2}Q_{\infty}[Q_{\infty}^{-1}Q,B^T]Q^{-1/2}=0.
\end{align*}
Since the matrix $D=Q^{1/2}Q_{\infty}^{-1}Q^{1/2}$ is invertible, it implies in particular that the following intersections of kernels reduces to 
$$\Big( \bigcap_{k=0}^{n-1}\textrm{Ker}
\big((A^2+\|A\|^2)D^{k}\big)\cap \textrm{Ker}
\big([A,D^k]\big)\Big)\cap \rr^{n}=\textrm{Ker}\big(A^2+\|A\|^2\big)\cap \rr^{n}.$$
While recalling from (\ref{skew}) that the matrix $A$ is skew-adjoint, it follows that 
$$\|A^2\|=\|A^TA\|=\|A\|^2,$$ 
is the spectral radius as well as an eigenvalue of the non-negative symmetric matrix $A^TA=-A^2$, implying that the following condition holds 
$$\Big( \bigcap_{k=0}^{n-1}\textrm{Ker}
\big((A^2+\|A\|^2)D^{k}\big)\cap \textrm{Ker}
\big([A,D^k]\big)\Big)\cap \rr^{n}=\textrm{Ker}\big(A^2+\|A\|^2\big)\cap \rr^{n} \neq \{0\}.$$
According to Theorem~\ref{th}, the spectral angle of any normal\footnote{i.e., when $[Q_{\infty}^{-1}Q,B^T]=0$} Ornstein-Uhlenbeck operator in $L^2_{\mu}$, and the angle of the analyticity sector of its associated semigroup are always complementary.

\bigskip

\noindent
\textit{Example 2.}
By coming back to the previous two-dimensional study case given in~\cite{sector} (Example~1), that is, when $Q=I_2$ and 
$$B=\left(\begin{array}{cc}
-\delta& \lambda \\
0 & -\delta      
\end{array}\right),$$
with $\delta>0$, $\lambda \in \rr^*$, we notice from (\ref{dan2}) that the Ornstein-Uhlenbeck operator is non-normal as 
$$[Q_{\infty}^{-1},B^T]=\frac{4\delta \lambda^2}{4\delta^2+\lambda^2}\left(\begin{array}{cc}
-\delta& 0 \\
\lambda & \delta      
\end{array}\right)\neq 0.$$
Furthermore, another direct computation shows that the matrix
$$[A,D]=2[Q_{\infty}B^T,Q_{\infty}^{-1}]=2(Q_{\infty}B^T-B^TQ_{\infty})Q_{\infty}^{-1}=\frac{\lambda^2}{2\delta^3}\left(\begin{array}{cc}
\delta & 0 \\
-\lambda & -\delta      
\end{array}\right)Q_{\infty}^{-1},$$
is invertible. It implies that $\textrm{Ker}\big([A,D]\big)=\{0\}$, therefore the condition (\ref{a21}) does not hold
$$\Big( \bigcap_{k=0}^{1}\textrm{Ker}
\big((A^2+\|A\|^2)D^{k}\big)\cap \textrm{Ker}
\big([A,D^k]\big)\Big)\cap \rr^{2}=\{0\}.$$
Thus, we recover from Theorem~\ref{th} that the spectral angle of this Ornstein-Uhlenbeck operator in $L^2_{\mu}$ and the angle of the analyticity sector of its associated semigroup are not complementary.

\bigskip

\noindent
\textit{Example 3.}
Let us consider the case when $Q=I_3$ and 
$$B=\left(\begin{array}{ccc}
-1& 1 & 0 \\
0 & -1 & 1 \\
0 & 0 & -1     
\end{array}\right).$$
The drift matrix $B$ satisfies the spectral condition (\ref{pav1}) as $\sigma(B)=\{-1\}$.
In this case, direct computations provide that 
$$Q_{\infty}=\int_0^{+\infty}e^{-2s}
\left(\begin{array}{ccc}
1& s & \frac{s^2}{2} \\
0 & 1 & s \\
0 & 0 & 1     
\end{array}\right)\left(\begin{array}{ccc}
1& 0 & 0 \\
s & 1 & 0 \\
\frac{s^2}{2} & s & 1     
\end{array}\right)
ds=\frac{1}{16}\left(\begin{array}{ccc}
15& 7 & 2 \\
7 & 12 & 4 \\
2 & 4 & 8     
\end{array}\right).$$
The Ornstein-Uhlenbeck operator is non-normal as 
$$[Q_{\infty}^{-1},B^T]=Q_{\infty}^{-1}[B^T,Q_{\infty}]Q_{\infty}^{-1}\neq 0,$$
as
$$[B^T,Q_{\infty}]=\frac{1}{16}\left(\begin{array}{ccc}
-7& -2 & 0 \\
3 & 3 & 2 \\
3 & 4 & 4     
\end{array}\right).$$
Furthermore, another direct computation provides
$$[A,D]=2[Q_{\infty}B^T,Q_{\infty}^{-1}]=-2[B^T,Q_{\infty}]Q_{\infty}^{-1}=
-\frac{1}{8}\left(\begin{array}{ccc}
-7& -2 & 0 \\
3 & 3 & 2 \\
3 & 4 & 4     
\end{array}\right)Q_{\infty}^{-1}.$$
As $\det([A,D])=\frac{1}{32}\det(Q_{\infty}^{-1})\neq 0$, the matrix $[A,D]$ is invertible and its kernel is zero $\textrm{Ker}\big([A,D]\big)=\{0\}$. The condition (\ref{a21})
therefore does not hold
$$\Big( \bigcap_{k=0}^{2}\textrm{Ker}
\big((A^2+\|A\|^2)D^{k}\big)\cap \textrm{Ker}
\big([A,D^k]\big)\Big)\cap \rr^{3}=\{0\}.$$
Indeed, we deduce from (\ref{rt4}) and (\ref{dan3}) that the spectral angle of this Ornstein-Uhlenbeck operator is zero $\theta_{\textrm{spec}}=0$ as
$$\sigma(B)=\{-1\} \subset \rr_-.$$
On the other hand, the matrix
$$I_3+2Q_{\infty}B^T=\frac{1}{8}\left(\begin{array}{ccc}
0& -5 & -2 \\
5 & 0 & -4 \\
2 & 4 & 0     
\end{array}\right),$$
is non-zero. We deduce from (\ref{rt2}) that the angle $0<\theta_2 \leq \frac{\pi}{2}$ cannot take the value $\frac{\pi}{2}$ as
$$\textrm{cotan }\theta_2=\|I_3+2Q_{\infty}B^T\|>0.$$ 
The spectral angle of this Ornstein-Uhlenbeck operator in $L^2_{\mu}$ and the angle of the analyticity sector of its associated semigroup are indeed not complementary
$$0<\theta_2 +\theta_{\textrm{spec}}=\theta_2< \frac{\pi}{2}.$$

\bigskip

As seen above, the spectral angle of a Ornstein-Uhlenbeck operator $P$ and the angle of the analyticity sector of its associated semigroup may fail to be complementary
$$0<\theta_2+\theta_{\textrm{spec}} < \frac{\pi}{2},$$
only when $P$ is non-normal when acting on $\sqrt{\rho}^{-1}\mathscr{S}(\rr^n)$, but the condition (\ref{a21}) is not equivalent to the non-normality property $[Q_{\infty}^{-1}Q,B^T]\neq 0$, and condition (\ref{a21})
can actually hold for some non-normal Ornstein-Ulhenbeck operators.

Let us now explain how the pseudospectral properties of non-normal Ornstein-Ulhenbeck operators account for the complementary angle property to hold, or not. Building on general results established in~\cite{duke} for elliptic quadratic operators, the study of resolvent estimates for non-normal Ornstein-Ulhenbeck operators was investigated in the work~\cite{OPPS2}. If $Q \gg 0$ is a symmetric positive definite matrix and $B \in M_n(\rr)$ satisfies the spectral condition (\ref{pav1}) together with $[Q_{\infty}^{-1}Q,B^T]\neq 0$, the result of Theorem~2.4 in~\cite{OPPS2} shows that the resolvent of the Ornstein-Uhlenbeck operator $P$ exhibits rapid growth in norm along all half-lines starting from $-\frac{1}{2}\textrm{Tr}(B)$ and contained in the particular open angular sector\footnote{The notation used for this angular sector will make sense afterwards as it will come up as the interior of the numerical range $\Sigma(q)$ of a particular quadratic symbol $q$.} 
$$-\frac{1}{2}\textrm{Tr}(B)-\stackrel{\circ}{\Sigma(q)},$$ 
with
$$ \stackrel{\circ}{\Sigma(q)}=\{z=t(1+i\lambda) \in \cc^* : \  m_- < \lambda < m_+, \ t>0\},$$
that is
$$\forall m_- < \lambda < m_+, \forall N \in \mathbb{N}, \quad \lim_{t \to +\infty}\frac{\big\|\big(P+\frac{1}{2}\textrm{Tr}(B)+t(1+i\lambda)\big)^{-1}\big\|_{\mathcal{L}(L^2_{\mu})}}{t^N}=+\infty,$$
where $\|\cdot\|_{\mathcal{L}(L_{\mu}^2)}$ denotes the operator norm in the space of bounded operators on $L^2_{\mu}=L^2(\rr^n,d\mu)$, with
$$m_-=\inf_{(x,\xi) \in \mathcal{S}}\Big\langle\Big(-\frac{1}{2}QQ_{\infty}^{-1}-B\Big)x,\xi\Big\rangle, \quad m_+=\sup_{(x,\xi) \in \mathcal{S}}\Big\langle\Big(-\frac{1}{2}QQ_{\infty}^{-1}-B\Big)x,\xi\Big\rangle,$$
and
$$\mathcal{S}=\Big\{(x,\xi) \in \rr^{2n} :\ \frac{1}{2}|Q^{\frac{1}{2}}\xi|^2+\frac{1}{8}|Q^{\frac{1}{2}}Q_{\infty}^{-1}x|^2=1\Big\}.$$
\begin{figure}[ht]\label{fig2s}
\caption{The resolvent $\|(P-z)^{-1}\|_{\mathcal{L}(L^2_{\mu})}$ blows-up along all the half-lines contained in the dark region starting from $-\frac{1}{2}\textrm{Tr}(B)$. The spectrum is represented with black dots.}
\centerline{\includegraphics[scale=0.65]{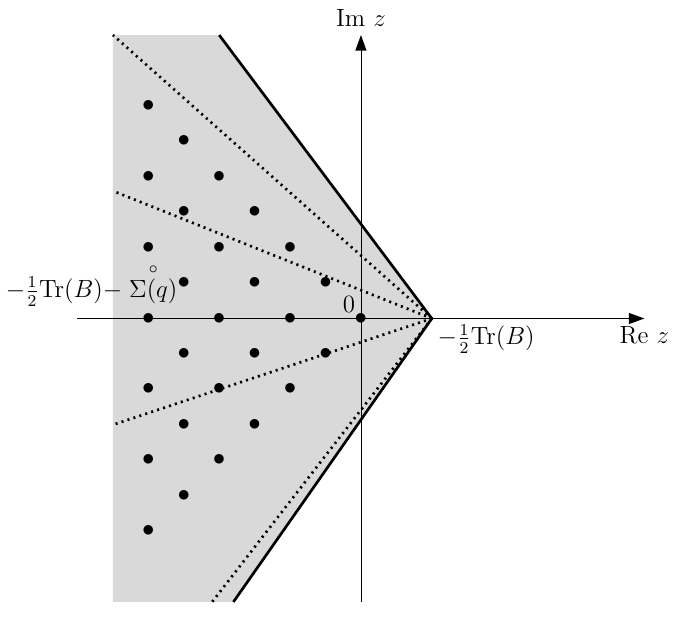}}
\end{figure}
\hspace{-3mm}
According to the characterization (\ref{dani1}), these phenomena of blow-up in norm along all the half-lines contained in the dark region starting from $-\frac{1}{2}\textrm{Tr}(B)$ appearing in Figure~\ref{fig2s} limit the angle of the analyticity sector of the associated semigroup. On the other hand, this dark region contains the spectrum of the operator $P$. However, the spectrum of the operator can spread out more or less deeply into this dark angular sector. As we shall see in the following, the condition (\ref{a21}) in Theorem~\ref{th} comes up as a necessary and sufficient condition for the spectral angle of the operator to be equal to half the opening of the dark angular sector\footnote{It will be shown that this dark angular sector is symmetric with respect to the real axis, i.e. $m_-=-m_+$.}. This condition will prevent the resolvent of the operator to blow-up far from its spectrum. It will ensure that the complementary angle property holds.

We close this introduction by describing the organization of the paper. The next section (Section~\ref{quad}) is devoted to the study of general elliptic accretive quadratic operators. We establish some results about the localization of their spectra. We also prove that they generate analytic semigroups and specify the angles of their analyticity sectors. Following and building on the analysis led in~\cite{OPPS2}, these results are then instrumental for the study of Ornstein-Uhlenbeck operators and allow to prove Theorem~\ref{th} in Section~\ref{uhl}. The last section is an appendix devoted to provide the proof of a technical lemma.

\bigskip

\noindent
\textbf{AI tools disclosure.} The author certifies that no AI assistance of any kind was used in the whole process for the elaboration, nor the writing of the present work. The refusal of any IA assistance is motivated by environmental concerns as well as personal aesthetic of what should be mathematics inspired in many ways  by the following title of the book by Dieudonn\'e: ``\textit{Pour l'honneur de l'esprit humain. Les math\'ematiques d'aujourd'hui}''.  

\bigskip

\noindent
{\bf Acknowledgements.}
The figures in the article have been created using the software Fig4Tex developed by Lafranche and Martin~\cite{lafranche}.

\section{Elliptic quadratic operators}\label{quad}

This section is devoted to recall few facts and establish some properties of general quadratic operators that will be instrumental to the present study of Ornstein-Uhlenbeck operators.

\subsection{Quadratic operators and notion of singular space}\label{secqw1}
Quadratic operators are pseudodifferential operators
$$q^w(x,D_x)u(x) =\frac{1}{(2\pi)^n}\iint_{\R^{2n}}{e^{i(x-y) \cdot \xi}q\Big(\frac{x+y}{2},\xi\Big)u(y)dyd\xi},$$
defined by the Weyl quantization of complex-valued quadratic symbols 
\begin{align*}
q: \rr^{n} \times \rr^n &\rightarrow \cc, \\
(x,\xi) &\mapsto q(x,\xi),
\end{align*}
with $n\geq 1$.
These non-selfadjoint operators are in fact only differential operators since the Weyl quantization of the quadratic symbol
$x^{\alpha} \xi^{\beta}$, with $(\alpha,\beta) \in \N^{2n}$, $|\alpha+\beta|=2$, is simply given by
\begin{equation}\label{forsym}
(x^{\alpha} \xi^{\beta})^w=\textrm{Op}^w(x^{\alpha} \xi^{\beta})=\frac{x^{\alpha}D_x^{\beta}+D_x^{\beta} x^{\alpha}}{2}, 
\end{equation}
with $D_x=i^{-1}\partial_x$. The maximal closed realization of a quadratic operator $q^w(x,D_x)$ on $L^2(\rr^n)$, that is, the operator equipped with the domain
\begin{equation}\label{dom1}
D(q)=\big\{u \in L^2(\rr^n) : \ q^w(x,D_x)u \in L^2(\rr^n)\big\},
\end{equation}
where $q^w(x,D_x)u$ is defined in the distribution sense, is known to coincide with the graph closure of its restriction to the Schwartz space~\cite{mehler} (pp.~425-426),
\begin{equation}\label{gra}
q^w(x,D_x) : \mathscr{S}(\rr^n) \rightarrow \mathscr{S}(\rr^n).
\end{equation}
Associated to a quadratic symbol $q : \rr_x^n \times \rr_{\xi}^n \rightarrow \mathbb{C}$ is its numerical range $\Sigma(q)$ defined as the closure in the
complex plane of all its values
$$\Sigma(q)=\overline{q(\rr_x^n \times \rr_{\xi}^n)} \subset \cc.$$
Another classical object associated to a quadratic symbol is the notion of Hamilton map, also known as the fundamental matrix, which is defined as the unique matrix satisfying the identity
$$\forall  (x,\xi) \in \R^{2n},\forall (y,\eta) \in \R^{2n}, \quad q\big((x,\xi),(y,\eta)\big)=\sigma\big((x,\xi),F(y,\eta)\big),$$
with $q(\cdot,\cdot)$ the polarized form associated to the quadratic form $q$, where $\sigma$ stands for the standard symplectic form
$$\sigma\big((x,\xi),(y,\eta)\big)=\langle \xi, y \rangle -\langle x, \eta\rangle=\sum_{j=1}^n(\xi_j y_j-x_j \eta_j),$$
with $x=(x_1,...,x_n)$, $y=(y_1,....,y_n)$, $\xi=(\xi_1,...,\xi_n)$, $\eta=(\eta_1,...,\eta_n) \in \cc^n$. A Hamilton map is always skew-symmetric with respect to the symplectic form. This is a consequence of the
properties of skew-symmetry of the symplectic form, and symmetry of the polarized form
$$\forall X,Y \in \rr^{2n}, \quad \sigma(X,FY)=q(X;Y)=q(Y;X)=\sigma(Y,FX)=-\sigma(FX,Y).$$
According to the definition, the Hamilton map expresses as 
$$F=\frac{1}{2}\left(\begin{array}{cc}
\nabla_{\xi}\nabla_x q & \nabla_{\xi}^2q  \\
-\nabla_x^2q & -\nabla_{x}\nabla_{\xi} q 
\end{array} \right),$$
where the matrices $\nabla_x^2q=(a_{i,j})_{1 \leq i,j \leq n}$,  $\nabla_{\xi}^2q=(b_{i,j})_{1 \leq i,j \leq n}$, $\nabla_{\xi}\nabla_x q =(c_{i,j})_{1 \leq i,j \leq n}$,
$\nabla_{x}\nabla_{\xi} q=(d_{i,j})_{1 \leq i,j \leq n}$ are given by
$$a_{i,j}=\partial_{x_i,x_j}^2 q, \quad b_{i,j}=\partial_{\xi_i,\xi_j}^2q, \quad c_{i,j}=\partial_{\xi_i,x_j}^2q, \quad d_{i,j}=\partial_{x_i,\xi_j}^2q.$$
It also follows from the definition that the real and imaginary parts of the Hamilton map~$F$,
$$\textrm{Re } F=\frac{1}{2}(F+\overline{F}), \quad \textrm{Im }F=\frac{1}{2i}(F-\overline{F}),$$ 
are the Hamilton maps associated respectively to the real and imaginary parts of the quadratic symbol $q$.

Since the classical works by Sj\"ostrand~\cite{sjostrand} and Boutet de Monvel~\cite{boutet3}, a complete description for the spectrum of elliptic quadratic operators is known and has played an important role in the analysis of partial differential operators with double characteristics. 
Elliptic quadratic operators are quadratic operators whose Weyl symbols satisfy to the ellipticity condition
\begin{equation}\label{yo1}
(x,\xi) \in \rr^{2n}, \ q(x,\xi)=0 \Rightarrow (x,\xi)=0.
\end{equation} 
When this ellipticity condition holds, the numerical range $\Sigma(q)$ has a specific shape \cite{sjostrand}. In dimension $n \geq 2$, one may always find a non-zero complex number $z \in \cc^*$ such that $\textrm{Re}(z q)$ is a positive definite quadratic form; whereas in dimension $n=1$, the same result holds true if in addition we assume that $\Sigma(q) \neq \cc$. The numerical range of an elliptic quadratic form can therefore take only two shapes: $\cc$, or a closed angular sector with a vertex in $0$ and an angle strictly less than $\pi$.
When the ellipticity condition (\ref{yo1}) holds, the domain of the quadratic operator $D(q)$ is equal to the Hilbert space
\begin{equation}\label{14.1}
B=\{ u \in L^2(\rr^n) : \ x^{\alpha} D_x^{\beta} u \in L^2(\rr^n), \ |\alpha+\beta| \leq 2\},
\end{equation}
equipped with the norm
$$\|u\|_B^2=\sum_{|\alpha+\beta| \leq 2}{\|x^{\alpha} D_x^{\beta} u\|_{L^2(\rr^n)}^2}.$$
The exact description of the spectrum of an elliptic quadratic operator
$$q^w(x,D_x) : B \rightarrow L^2(\rr^n),$$
whose numerical range satisfies $\Sigma(q) \neq \cc$, was given independently by Sjöstrand and Boutet de Monvel in the 1970s~\cite{boutet3,sjostrand}. This spectrum is only composed of eigenvalues
\begin{equation}\label{15}
\sigma(q^w(x,D_x))=\Big\{ \sum_{\substack{\lambda \in \sigma(F), \\  -i \lambda \in \Sigma(q)\setminus \{0\}} }
{(r_{\lambda}+2 k_{\lambda})(-i\lambda) :\ k_{\lambda} \in \mathbb{N}}
\Big\} \subset \Sigma(q), 
\end{equation}
where $F$ is the Hamilton map of the quadratic form $q$ and $r_{\lambda}$ is the dimension of the space of generalized eigenvectors of $F$ in $\cc^{2n}$ associated to the eigenvalue $\lambda \in \cc$.

In~\cite{HPS}, the notion of singular space was introduced by Hitrik and the author by pointing out the existence of a particular vector subspace in the phase space, which is intrinsically associated to a quadratic symbol $q$, and defined as the following finite intersection of kernels
\begin{equation}\label{singular}
S=\Big( \bigcap_{j=0}^{2n-1}\textrm{Ker}
\big[\textrm{Re }F(\textrm{Im }F)^j \big]\Big)\cap \rr^{2n},
\end{equation}
where $\textrm{Re }F$ and $\textrm{Im }F$ stand respectively for the real and imaginary parts of the Hamilton map $F$ associated to the quadratic symbol $q$.
As pointed out in \cite{HPS,OPPS,karel,viola1}, the singular space plays a basic role in understanding the spectral and hypoelliptic properties of non-elliptic quadratic operators, as well as the spectral and pseudospectral properties of certain classes of degenerate doubly characteristic pseudodifferential operators~\cite{kps3,kps4,viola}. Some applications of these results to the study of degenerate hypoelliptic Ornstein-Uhlenbeck operators and degenerate hypoelliptic Fokker-Planck operators are also given in~\cite{OPPS2}, and play a key role in the present work.

More specifically, the class of quadratic operators $q^w(x,D_x)$ whose symbols have non-negative real parts $\textrm{Re }q \geq 0$, and satisfy to a condition of partial ellipticity on the phase space, namely a condition of ellipticity on their singular space 
\begin{equation}\label{sm5}
(x,\xi) \in S, \ q(x,\xi)=0 \Rightarrow (x,\xi)=0,
\end{equation}
is studied in~\cite{HPS}. 
This condition of partial ellipticity is obviously weaker than the ellipticity condition (\ref{yo1}) since $S \subset \rr^{2n}$, and this strictly as soon as $S \neq \rr^{2n}$, that is, when the real part of the symbol is not identically equal to zero. The result of \cite{HPS} (Theorem~1.2.2) shows that the description of the spectrum known in the elliptic case (\ref{15}) extends to non-elliptic quadratic operators whose symbols have a non-negative real part and enjoy ellipticity on their singular spaces as follows.
Let $q : \rr_x^n \times \rr_{\xi}^n \rightarrow \cc$, $n \geq 1$, be a complex-valued quadratic form with a non-negative real part $\textrm{Re }q \geq 0$. When the quadratic symbol $q$ is elliptic on its singular space
$$(x,\xi) \in S, \ q(x,\xi)=0 \Rightarrow (x,\xi)=0,$$ 
the spectrum of the quadratic operator  $q^w(x,D_x)$ is only composed by eigenvalues
\begin{equation}\label{sm6}
\sigma(q^w(x,D_x))=\Big\{ \sum_{\substack{\lambda \in \sigma(F), \\  -i \lambda \in \cc_+
\cup (\Sigma(q|_S) \setminus \{0\})} }
{\big{(}r_{\lambda}+2 k_{\lambda}
\big{)}(-i\lambda) :\ k_{\lambda} \in \mathbb{N}}
\Big\}, 
\end{equation}
where $F$ is the Hamilton map of the quadratic form $q$, $r_{\lambda}$ is the dimension of the space of generalized eigenvectors of $F$ in $\cc^{2n}$ associated to the eigenvalue $\lambda \in \cc$,
$$\Sigma(q|_S)=\overline{q(S)}, \quad  \cc_+=\{z \in \cc : \textrm{Re }z>0\}.$$
The proof of this result takes advantage of a geometric feature uncovered in~\cite{HPS}, which is that the singular space $S$ associated to a complex-valued quadratic form $q$ has always a symplectic structure\footnote{In the sense that the restriction of the symplectic form $\sigma|_{S}$ to the singular space is non-degenerate.} when the symbol $q$ is elliptic on $S$. This feature allows to decompose the phase space into the direct sum of the two symplectically orthogonal spaces
$$\rr^{2n}=S\oplus^{\sigma \perp} S^{\sigma \perp}, \quad \textrm{dim}_{\rr}S^{\sigma \perp}=2n', \quad  \textrm{dim}_{\rr}S=2n'', \quad n=n'+n'',$$
where $S^{\sigma \perp}$ is the orthogonal complement of the singular space in $\rr^{2n}$ with respect to the symplectic form 
$$S^{\sigma \perp}=\{X \in \rr^{2n} : \ \forall Y \in S, \ \sigma(X,Y)=0\}.$$
By using some stability properties induced directly by the definition of the singular space together with the Cayley-Hamilton Theorem
$$(\textrm{Re }F) S \subset S, \quad (\textrm{Im }F)S \subset S, \quad (\textrm{Re }F) S^{\sigma \perp} \subset S^{\sigma \perp}, \quad (\textrm{Im }F)S^{\sigma \perp} \subset S^{\sigma \perp},$$
as 
$$(\textrm{Re }F) S=\{0\} \quad \textrm{ and } \quad \forall j \geq 0, \quad (\textrm{Im }F)^j \in \textrm{Vect}_{\rr}(I,\textrm{Im }F,...,(\textrm{Im }F)^{2n-1}),$$ 
the variables in $S$ and $S^{\sigma \perp}$ can be tensorized, and some symplectic coordinates onto the symplectic spaces $S$ and $S^{\sigma \perp}$, 
$$X=X'+X'', \quad X=(x,\xi) \in \rr^{2n}, \ X'=(x',\xi') \in S^{\sigma \perp}, \ X''=(x'',\xi'') \in S,$$
can be chosen so that
$$q(X)=q|_{S^{\sigma \perp}}(X')+q|_S(X''),$$
where the second quadratic form $q|_S$ is purely imaginary-valued, i.e., equal to
$$q|_S=i \tilde{q}|_S,$$
with $\tilde{q}|_S$ a real-valued quadratic form, whereas the first quadratic form $q|_{S^{\sigma \perp}}$ is a complex-valued quadratic form with a non-negative real part $\textrm{Re }q|_{S^{\sigma \perp}} \geq 0$.
The assumption of ellipticity on the singular space imply that $\tilde{q}|_S$ is an elliptic real-valued quadratic form, i.e., that there exists $\eps_0 \in \{\pm 1\}$ such that $\eps_0 \tilde{q}|_S$ is a positive definite quadratic form. Up to a new choice of symplectic coordinates, any positive definite quadratic form can be reduced to a harmonic oscillator thanks to~\cite{hormander} (Lemma~18.6.4). In such coordinates, the restriction of the quadratic form $q|_S$ on the singular space reduces to the normal form (see Proposition~2.0.1 in~\cite{HPS}),
\begin{equation}\label{evg40}
q|_S(X'')=i\eps_0 \sum_{j=1}^{n''}{\lambda_j(\xi_j''^2+x_j''^2)}, \quad X''=(x'',\xi'') \in S,
\end{equation}
with $\lambda_j>0$. The formula (\ref{evg40}) implies in particular that the set $\Sigma(q|_S)  \setminus \{0\}$ appearing in the formula (\ref{sm6}) satisfies
$$\Sigma(q|_S)  \setminus \{0\}=\overline{q(S)} \setminus \{0\}\subset i\eps_0\rr_+^*.$$
Regarding the second quadratic form $q|_{S^{\sigma \perp}}$, its real part fails in general to be positive definite unless the real part of $q$ is. However, this quadratic form $q|_{S^{\sigma \perp}}$ enjoys specific dynamical features. Indeed, as proved in~\cite{HPS} (Proposition~2.0.1), the definition of the singular space implies the following properties:
\begin{itemize}
\item[$(i)$] The average of the real part of
$q|_{S^{\sigma \perp}}$ by the flow generated by the Hamilton vector field associated to its imaginary part
$H_{\textrm{Im}q|_{S^{\sigma \perp}}}$,
$$\langle \textrm{Re }q|_{S^{\sigma \perp}}\rangle_T(X')=\frac{1}{2T}\int_{-T}^T{(\textrm{Re }q|_{S^{\sigma \perp}})(e^{tH_{\textrm{Im}q|_{S^{\sigma \perp}}}}X')dt} \gg 0, \quad T>0,$$
is a positive definite quadratic form on the symplectic space $S^{\sigma \perp}$;
\item[$(ii)$] The sum of the non-negative quadratic forms
$$\sum_{j=0}^{2n'-1}{(\textrm{Re } q|_{S^{\sigma \perp}})((\textrm{Im } F|_{S^{\sigma \perp}})^j X')} \gg 0,$$
is positive definite on the symplectic space $S^{\sigma \perp}$.
\end{itemize}
Following~\cite{HPS} (p. 807), the singular space can also be described in terms of the eigenspaces of $F$
associated to its real eigenvalues. Indeed, when $\textrm{Re }q \geq 0$ and $q$ is elliptic on it singular space, the set of real eigenvalues of the Hamilton map $F$ writes as
$$\sigma(F) \cap \rr =\{\mu_1,...,\mu_r,-\mu_1,...,-\mu_r\},$$
with $\mu_j > 0$, $\mu_j \neq \mu_k$  if $j \neq k$, and $\{\mu_1,...,\mu_r\}=\{\lambda_1,...,\lambda_{n''}\}$, where $\lambda_j>0$ are the positive parameters appearing in the formula (\ref{evg40}).
The singular space is then the direct sum of the symplectically orthogonal spaces
$$S=S_{\mu_1} \oplus^{\sigma \perp} S_{\mu_2} \oplus^{\sigma \perp}... \oplus^{\sigma \perp} S_{\mu_r},$$
where $S_{\mu_{j}}$ is the symplectic space
$$S_{\mu_j}=\big(\textrm{Ker}(F -\mu_j I) \oplus \textrm{Ker}(F+\mu_j I) \big) \cap \rr^{2n}.$$
It follows that the set of real eigenvalues of  $F$ is non-empty
$$\sigma(F) \cap \rr \neq \emptyset,$$
if and only if the singular space $S$ is non-zero, i.e. $S \neq \{0\}$. We can therefore deduce from (\ref{sm6}) the following the result:

\bigskip

\begin{proposition}\label{prop4}
Let $q : \rr^{n} \times \rr^n \rightarrow \cc$, with $n \geq 1$, be a complex-valued quadratic form with a non-negative real part $\emph{\textrm{Re }}q \geq 0$. If the quadratic symbol $q$ is elliptic on its singular space $S$, i.e. when the condition \emph{(\ref{sm5})} holds, the spectrum of quadratic operator $q^w(x,D_x)$ contains some eigenvalues regularly spaced out along an half-line parallel to the imaginary axis in the complex plane if and only if the singular space $S$ is non-zero, i.e. $S\neq \{0\}$.
\end{proposition}

\begin{figure}[ht]
\caption{Spectrum $\sigma(q^w(x,D_x))$: Case when $S \neq \{0\}$ (Left); Case when $S=\{0\}$ (Right).}
\begin{minipage}[b]{0.3\linewidth}
\includegraphics[scale=0.34]{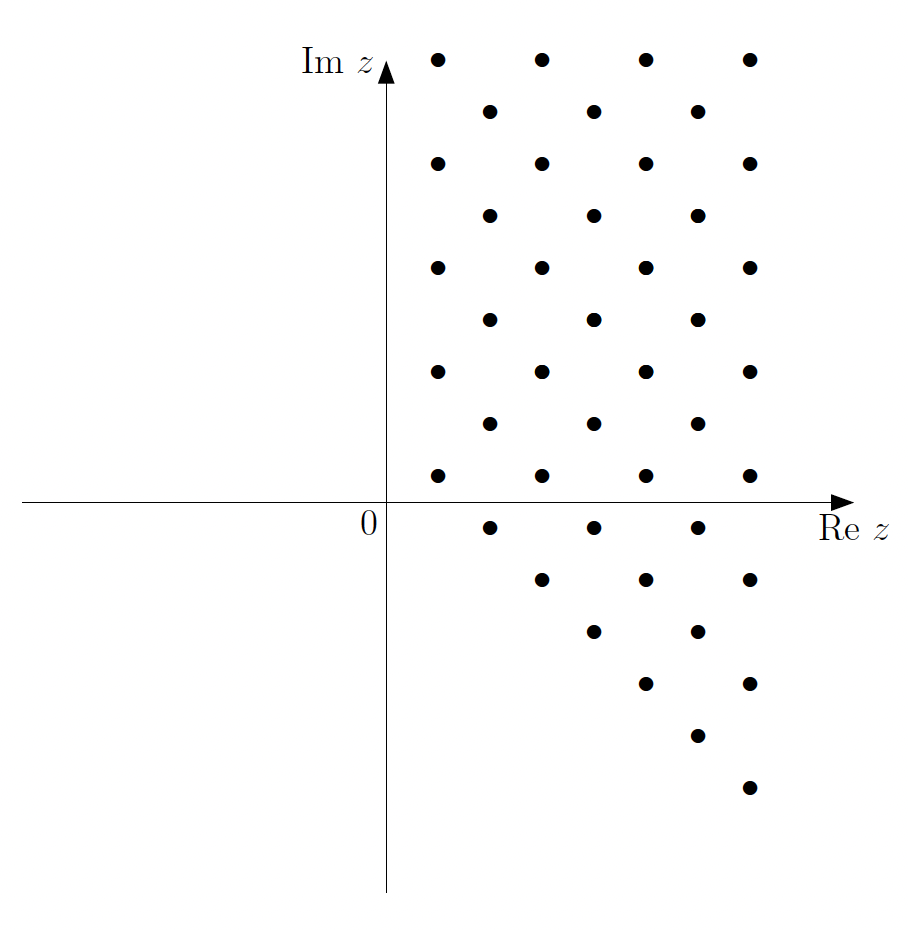}
\end{minipage}
\hspace{2cm}
\begin{minipage}[b]{0.3\linewidth}
\includegraphics[scale=0.2]{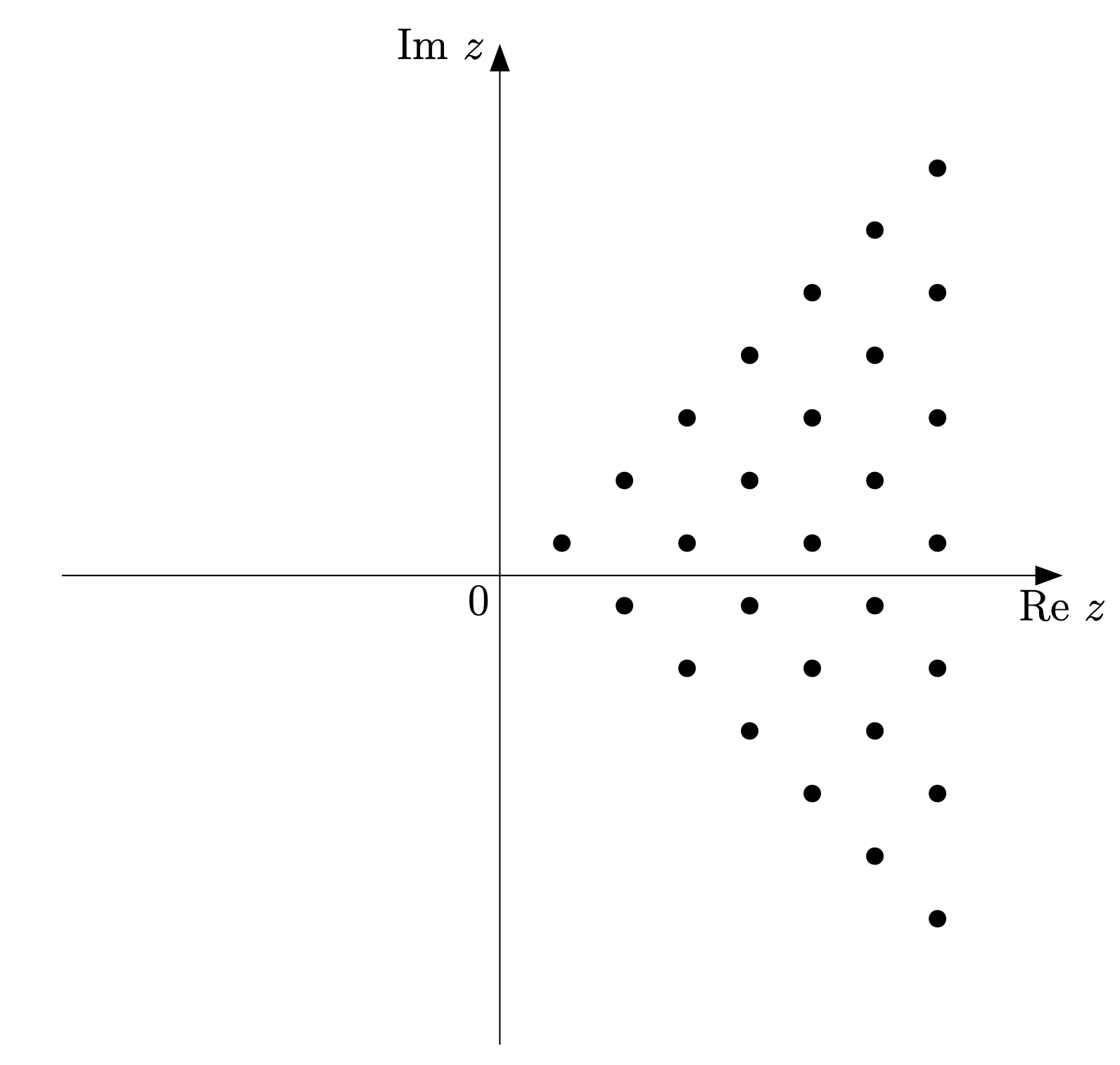}
\end{minipage}
\hspace{0.5cm}
\end{figure}

\bigskip

\subsection{Analytic semigroups generated by elliptic quadratic differential operators}

Let $q : \rr^{n}_x \times \rr_{\xi}^n \rightarrow \cc$, with $n \geq 1$, be a quadratic form with a non-negative real part $\textrm{Re }q \geq 0$. We know from~\cite{mehler} (pp. 425-426) that the quadratic operator $q^w(x,D_x)$ equipped with the domain (\ref{dom1}), is maximal accretive and generates a contraction semigroup $(e^{-tq^w})_{t \geq 0}$ on $L^2(\rr^n)$.

In the following, we consider the specific case when the real part $\textrm{Re }q \gg 0$ is a positive definite quadratic form, which is a special case of elliptic quadratic symbol
$$(x,\xi) \in \rr^{2n}, \ q(x,\xi)=0 \Rightarrow (x,\xi)=0.$$
According to the results recalled above (Subsection~\ref{secqw1}), the numerical range $\Sigma(q)$ is then a closed angular sector with a vertex in $0$ and an angle strictly less than $\pi$ included in $\cc_+ \cup \{0\}$,
$$\Sigma(q)=\overline{q(\rr_x^n \times \rr_{\xi}^n)} \subset \cc_+\cup \{0\}=\{z \in \cc : \textrm{Re }z>0\}\cup \{0\}.$$
It therefore exist some $-\frac{\pi}{2} < \theta_- \leq \theta_+ <\frac{\pi}{2}$ such that 
$$\partial\Sigma(q)=e^{i\theta_-}\rr_+ \cup e^{i\theta_+}\rr_+.$$
The next result shows that $(e^{-tq^w})_{t \geq 0}$ is an analytic semigroup on $L^2(\rr^n)$:

\bigskip

\begin{proposition}\label{propo2}
Let $q : \rr_x^{n} \times \rr_{\xi}^n \rightarrow \cc$, with $n \geq 1$, be a complex-valued quadratic form with a positive definite real part $\emph{\textrm{Re }}q \gg 0$. 
Then, the quadratic operator $q^w(x,D_x)$ generates a contraction semigroup $(e^{-tq^w})_{t \geq 0}$ on $L^2(\rr^n)$, which is analytic in the angular sector
$$\Sigma_{\delta}=\{z \in \cc^* : |\arg z| <\delta\} \cup \{0\},$$
with 
$$\partial\Sigma(q)=e^{i\theta_-}\rr_+ \cup e^{i\theta_+}\rr_+ ; \quad -\frac{\pi}{2} < \theta_- \leq \theta_+ <\frac{\pi}{2} ; \quad 0<\delta=\inf\Big(\frac{\pi}{2}-\theta_+, \theta_-+\frac{\pi}{2}\Big)\leq \frac{\pi}{2}.$$
\end{proposition}

\bigskip

\begin{proof}
Let $\tilde{q} : \rr_x^{n} \times \rr_{\xi}^n \rightarrow \cc$, with $n \geq 1$, be an arbitrary complex-valued quadratic form with a non-negative real part $\textrm{Re }\tilde{q} \geq 0$. We deduce from~\cite{hormander} (Theorem~21.5.3) that there exists a real linear symplectic transformation $\chi$ of $\rr^{2n}$ such that 
$$\textrm{Re}(\tilde{q} \circ \chi)(x,\xi)=\sum_{j=1}^{k}{\lambda_j(\xi_j^2+x_j^2)}
+\sum_{j=k+1}^{k+l}{x_j^2},$$
with $k,l \in \nn$ and $\lambda_j>0$ for all $j=1,...,k$. The symplectic invariance of the Weyl quantization~\cite{hormander} (Theorem~18.5.9) allows to find a metaplectic operator $U$, which is a unitary transformation on $L^2(\rr^n)$, and an automorphism of the spaces $\mathscr{S}(\rr^n)$ and $B$, where $B$ is the Hilbert space defined in (\ref{14.1}) such that  
$$(\textrm{Re } \tilde{q})^w(x,D_x)=U^{-1} \Big(\sum_{j=1}^{k}{\lambda_j(D_{x_j}^2+x_j^2)}
+\sum_{j=k+1}^{k+l}{x_j^2} \Big) U.$$
Thanks to the standard properties of the Weyl quantization 
$$\textrm{Re}(\tilde{q}^w(x,D_x))=\frac{1}{2}\big(\tilde{q}^w(x,D_x)+\tilde{q}^w(x,D_x)^*\big)=(\textrm{Re }\tilde{q})^w(x,D_x),$$ 
we obtain that for all $z \in \cc$ and $u \in \mathscr{S}(\rr^n)$,
\begin{multline*}
\textrm{Re}(\tilde{q}^w(x,D_x)u-z u,u)_{L^2}= (\textrm{Re}(\tilde{q}^w(x,D_x))u,u)_{L^2} -(\textrm{Re }z) \|u\|_{L^2}^2 \\ 
=((\textrm{Re }\tilde{q})^w(x,D_x)u,u)_{L^2} -(\textrm{Re }z) \|u\|_{L^2}^2.
\end{multline*} 
By using that $U$ is a unitary operator on $L^2(\rr^{n})$, we obtain that for all $z \in \cc$ and $u \in \mathscr{S}(\rr^n)$,
\begin{multline}\label{inf3.52}
\|\tilde{q}^w(x,D_x)u-z u\|_{L^2}\|u\|_{L^2} \geq \textrm{Re}(\tilde{q}^w(x,D_x)u-z u,u)_{L^2}\\
= \sum_{j=1}^{k}{\lambda_j (\|D_{x_j} U u\|_{L^2}^2 +\|x_j U u\|_{L^2}^2 )} +\sum_{j=k+1}^{k+l}{\|x_j U u\|_{L^2}^2} -(\textrm{Re }z) \|u\|_{L^2}^2
\geq  -(\textrm{Re }z) \|u\|_{L^2}^2. 
\end{multline} 
By applying the estimate (\ref{inf3.52}) to the three quadratic forms $q$ and $q_{\pm}=e^{\pm i(\frac{\pi}{2}-\theta_{\pm})}q$, whose real parts are non-negative, we derive that
\begin{equation}\label{inf3.53}
\forall z \in \cc, \forall u \in \mathscr{S}(\rr^n), \quad   d(z,\Sigma(q)) \|u\|_{L^2} \leq \|q^w(x,D_x)u-z u\|_{L^2}, 
\end{equation} 
where $d(z,\Sigma(q))$ is the distance of $z$ to the closed convex subset $\Sigma(q)$. According to (\ref{gra}) and (\ref{15}), we deduce from (\ref{inf3.53}) that 
$$\forall z \in \cc \setminus \Sigma(q), \quad   \|(q^w(x,D_x)-z)^{-1}\|_{\mathcal{L}(L^2)} \leq \frac{1}{d(z,\Sigma(q))}.$$ 
Let $0<\alpha<\delta$. 
When $z \in \cc^*$ with $\theta_++\frac{\pi}{2}\leq \textrm{arg }z \leq \pi$, or $-\pi < \textrm{arg }z \leq \theta_--\frac{\pi}{2}$, the projection of $z$ onto $\Sigma(q)$ is zero and $d(z,\Sigma(q))=|z|$. When $z \in \cc^*$ with $\theta_++\alpha\leq \textrm{arg }z \leq \theta_++\frac{\pi}{2}$, or $\theta_--\frac{\pi}{2}\leq \textrm{arg }z \leq \theta_--\alpha$, we have
$$\frac{|z|}{d(z,\Sigma(q))}\leq \frac{1}{\sin \alpha} <+\infty.$$
It follows that 
\begin{multline*}
\forall z \in \cc^*, \ \frac{\pi}{2}-(\delta-\alpha) \leq \textrm{arg }z \leq \pi, \textrm{ or } -\pi < \textrm{arg }z \leq -\frac{\pi}{2}+(\delta-\alpha), \\   \|(q^w(x,D_x)-z)^{-1}\|_{\mathcal{L}(L^2)} \leq \frac{1}{(\sin \alpha)|z|},
\end{multline*}
since 
$$\theta_++\alpha \leq \frac{\pi}{2}-(\delta-\alpha) \ \textrm{ and }\ -\frac{\pi}{2}+(\delta-\alpha) \leq \theta_--\alpha.$$
According to the characterization (\ref{dani1}), the semigroup $(e^{-tq^w})_{t \geq 0}$ is then analytic in the angular sector $\Sigma_{\delta}$.

\begin{figure}[ht]
\caption{$|z|\|(q^w(x,D_x)-z)^{-1}\|_{\mathcal{L}(L^2_{\mu})}$ remains bounded in the three lighter angular sectors of the figure.}
\centerline{\includegraphics[scale=0.65]{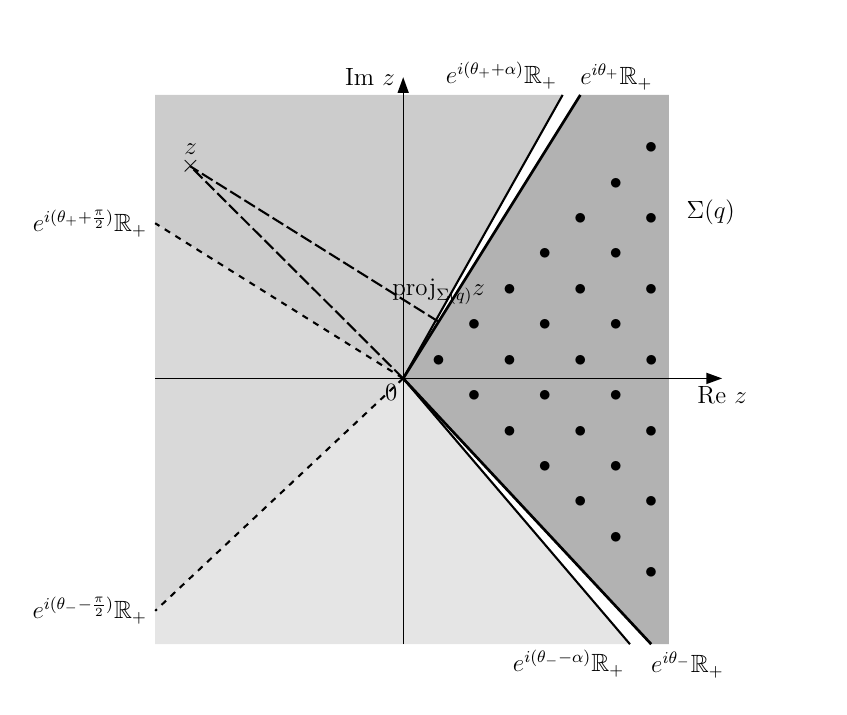}}
\end{figure}
\hspace{-2.5mm}

\end{proof}

\section{Ellipic Ornstein-Uhlenbeck operators}\label{uhl}
In this section, we follow the work~\cite{OPPS2} and make the link between the study of elliptic quadratic operators and the one of Ornstein-Uhlenbeck operators. 

Let $P$ be the Ornstein-Ulhenbeck operator defined in (\ref{jen0}), where $Q$ and $B$ are real $n \times n$-matrices, with $Q$ symmetric positive definite and $B$ satisfying the spectral condition (\ref{pav1}). We may associate to the operator $P$ acting on $L^2_{\mu}=L^2(\rr^n,d\mu)$, where $d\mu(x)=\rho(x)dx$ denotes the invariant measure with the density with respect to the Lebesgue measure (\ref{pav2}), the quadratic operator $\mathscr{L}$ acting on $L^2=L^2(\rr^n,dx)$,
\begin{equation}\label{pav4.5}
\mathscr{L}u=-\sqrt{\rho}P\big((\sqrt{\rho})^{-1}u\big)-\frac{1}{2}\textrm{Tr}(B)u.
\end{equation}
We notice that the localization of the spectrum (\ref{pav1}) implies that 
\begin{equation}\label{ai2}
\textrm{Tr}(B)<0,
\end{equation}
since $B \in M_n(\rr)$. 
Recalling the notation (\ref{not}), a direct computation led in~\cite{OPPS2} (Section~3.1) shows that
$$\mathscr{L} =\frac{1}{2}|Q^{1/2}D_x|^2+\frac{1}{8}|Q^{1/2}Q_{\infty}^{-1}x|^2-i\Big\langle\Big(\frac{1}{2}QQ_{\infty}^{-1}+B\Big)x,D_x\Big\rangle,$$ 
with $D_x=i^{-1}\nabla_x$, where $Q_{\infty}$ is the symmetric positive definite matrix (\ref{pav3}).
The operator $\mathscr{L}$ may be considered as a pseudodifferential operator
$$\mathscr{L}u=q^w(x,D_x)u(x) =\frac{1}{(2\pi)^n}\iint_{\R^{2n}}{e^{i(x-y) \cdot \xi}q\Big(\frac{x+y}{2},\xi\Big)u(y)dyd\xi},$$
defined by the Weyl quantization of the quadratic symbol
\begin{equation}\label{pav6}
q(x,\xi)=\frac{1}{2}|Q^{1/2}\xi|^2+\frac{1}{8}|Q^{1/2}Q_{\infty}^{-1}x|^2-i\Big\langle\Big(\frac{1}{2}QQ_{\infty}^{-1}+B\Big)x,\xi\Big\rangle, \quad (x,\xi) \in \rr^{2n}.
\end{equation}
This explicit computation is performed in~\cite{OPPS2} (Section~3.2), while using the formula (\ref{forsym}).  We notice that the real part of the symbol 
$$\forall (x,\xi) \in \rr^{2n} \setminus \{0\}, \quad (\textrm{Re }q)(x,\xi)=\frac{1}{2}|Q^{1/2}\xi|^2+\frac{1}{8}|Q^{1/2}Q_{\infty}^{-1}x|^2 > 0,$$ 
is a positive definite quadratic form, since $Q$, $Q_{\infty}$ and $B \in M_n(\rr)$, and $Q$, $Q_{\infty}$ are symmetric positive definite matrices. 
We check in~\cite{OPPS2} (Section~3.2) that the Hamilton map of the quadratic form (\ref{pav6}) is given by
\begin{equation}\label{F}
F=\lv\begin{array}{cc}
 -\frac{i}{4}(QQ_{\infty}^{-1}+2B)&\frac{1}{2}Q \\
 -\frac{1}{8}Q_{\infty}^{-1}QQ_{\infty}^{-1}  & \frac{i}{4}(QQ_{\infty}^{-1}+2B)^T       
\end{array}\rv.
\end{equation}
By using that the matrices $Q$, $Q_{\infty}$ and $B$ have real entries, we observe that the real part of the Hamilton map
\begin{equation}\label{nn1}
\textrm{Re }F=\frac{1}{8}\lv\begin{array}{cc}
 0&4Q \\
 -Q_{\infty}^{-1}QQ_{\infty}^{-1}  & 0    
\end{array}\rv,
\end{equation}
is invertible whose inverse is
\begin{equation}\label{ai9}
(\textrm{Re }F)^{-1}=\lv\begin{array}{cc}
 0&  -8Q_{\infty}Q^{-1}Q_{\infty} \\
2Q^{-1}  & 0    
\end{array}\rv.
\end{equation}
According to the definition (\ref{singular}), the singular space associated to the quadratic symbol $q$ is therefore equal to $S=\{0\}$.
It is shown in~\cite{OPPS2} (Corollary~3.5) that the spectrum of the Hamilton map $F$ is given by
\begin{equation}\label{azf2}
\sigma(F)=\sigma\Big(-\frac{i}{2}B\Big) \cup \sigma\Big(\frac{i}{2}B\Big),
\end{equation}
and more precisely that the algebraic multiplicity of $\mu$ as an eigenvalue of $F$ is equal to the sum of the algebraic multiplicity of $\mu$ as an eigenvalue of $\frac{i}{2}B$ and the algebraic multiplicity of $\mu$ as an eigenvalue of $-\frac{i}{2}B$. We deduce from (\ref{pav1}) that 
\begin{equation}\label{jk3}
\{\lambda \in \sigma(F) :  -i\lambda \in \CC_+\}=\Big\{\pm \frac{i}{2} \mu : \mu \in \sigma(B),\ \pm \mu \in \CC_+\Big\}
=\Big\{- \frac{i}{2} \mu : \mu \in \sigma(B)\Big\},
\end{equation}
since $\sigma(B) \subset \CC_-$. Besides, it follows that $r_{\lambda}$ the dimension of the space of generalized eigenvectors of $F$ in $\CC^{2n}$ associated to the eigenvalue $\lambda=- \frac{i}{2} \mu$, with $\mu \in \sigma(B)$, is exactly equal to $\tilde{r}_{\mu}$ the dimension of the space of generalized eigenvectors of $B$ in $\CC^{n}$ associated to the eigenvalue $\mu$, since
$$\sigma\Big(-\frac{i}{2}B\Big) \cap \sigma\Big(\frac{i}{2}B\Big)=\emptyset,$$
as $\sigma(B) \subset \CC_-$. By using that the mappings 
\begin{equation}\label{ai1}
\begin{array}{cc}
T :  L^2_{\mu} & \rightarrow  L^2\\
\ \ \ u &  \mapsto  \sqrt{\rho}u
\end{array}, \qquad \begin{array}{cc}
T^{-1} :  L^2 & \rightarrow L^2_{\mu}\\
\quad \ \ \ u & \mapsto \sqrt{\rho}^{-1}u
\end{array},
\end{equation}
are isometric isomorphisms, it allows to recover from (\ref{sm6}), (\ref{pav4.5}) and (\ref{jk3}) the description for the spectrum of the operator $P$ acting on $L_{\mu}^2$ given in (\ref{rt4}), which is only composed of eigenvalues
\begin{multline}\label{cx3.9}
\sigma(P)=-\frac{1}{2}\textrm{Tr}(B)-\sigma(\mathscr{L})=\Big\{ -\frac{1}{2}\textrm{Tr}(B)-\sum_{\substack{\lambda \in \sigma(F), \\  -i \lambda \in \cc_+} }
{\big{(}r_{\lambda}+2 k_{\lambda}
\big{)}(-i\lambda) :\ k_{\lambda} \in \mathbb{N}}
\Big\}\\ 
=\Big\{-\frac{1}{2}\textrm{Tr}(B)+\frac{1}{2}\sum_{\mu \in \sigma(B)} (\tilde{r}_{\mu}+2k_{\mu})\mu :\ k_{\mu} \in \mathbb{N}\Big\}=\Big\{\sum_{\mu \in \sigma(B)}\mu k_{\mu} :\ k_{\mu} \in \mathbb{N}\Big\},
\end{multline}
since $\Sigma(q|_S) \setminus \{0\}=\emptyset$, as the singular space is zero $S=\{0\}$, and
$$\sum_{\mu \in \sigma(B)}\mu \tilde{r}_{\mu}=\textrm{Tr}(B).$$
We deduce as well from (\ref{pav4.5}), (\ref{ai1}) and (\ref{cx3.9}) the following formula between the norms of the two resolvents
\begin{equation}\label{bn-1}
\forall z \not\in \sigma(\mathscr{L}), \quad  \|(\mathscr{L}-z)^{-1}\|_{\mathcal{L}(L^2)}=\Big\|\Big(P+\frac{1}{2}\textrm{Tr}(B)+z\Big)^{-1}\Big\|_{\mathcal{L}(L_{\mu}^2)},
\end{equation}
where $\|\cdot\|_{\mathcal{L}(L_{\mu}^2)}$ denotes the operator norm in the space of bounded operators on $L^2_{\mu}$. At the level of the semigroups, the following identity holds
\begin{equation}\label{bn1}
\forall t \geq 0, \forall v \in L^2_{\mu}, \quad T(t)v=e^{tP}v=e^{-\frac{t}{2}\textrm{Tr}(B)}(\sqrt{\rho})^{-1}e^{-t\mathscr{L}}(\sqrt{\rho}v).
\end{equation}
Thanks to the general Mehler formula proved by H\"ormander~\cite{mehler} (Theorem~4.1), it is actually possible to compute explicitly the Weyl symbol of the semigroup operator $e^{-t\mathscr{L}}$ at any time $t>0$, viewed as pseudodifferential operator, and to check that the formula (\ref{bn1}) allows to recover the explicit representation (\ref{gen}).
According to~\cite{mehler} (p.~426), the $L^2$-adjoint~$\mathscr{L}^*$ is then the quadratic operator with the Weyl symbol
$$\overline{q}(x,\xi)=\frac{1}{2}|Q^{1/2}\xi|^2+\frac{1}{8}|Q^{1/2}Q_{\infty}^{-1}x|^2+i\Big\langle\Big(\frac{1}{2}QQ_{\infty}^{-1}+B\Big)x,\xi\Big\rangle, \quad (x,\xi) \in \rr^{2n}.$$
According to (\ref{pav4.5}) and (\ref{ai2}), the operator $P$ is selfadjoint on $L^2_{\mu}$ if and only if the operator $\mathscr{L}$ is selfadjoint on $L^2(\rr^n)$, that is, if and only if $q=\overline{q}$, i.e. when
the condition 
$$\frac{1}{2}Q_{\infty}^{-1}Q+B^T=0,$$
holds.
We observe from (\ref{pav4.5}) that the operator $P$ commutes with its $L^2_{\mu}$-adjoint~$P^*$, 
$$\forall u \in \mathscr{S}(\rr^n), \quad [P,P^*](\sqrt{\rho}^{-1}u)=0,$$ 
when acting on $\sqrt{\rho}^{-1}\mathscr{S}(\rr^n)$,
if and only if the operator $\mathscr{L}$ commutes with its $L^2$-adjoint~$\mathscr{L}^*$,  
$$\forall u \in \mathscr{S}(\rr^n), \quad [\mathscr{L},\mathscr{L}^*]u=0,$$ 
when acting on $\mathscr{S}(\rr^n)$, since $[\mathscr{L},\mathscr{L}^*]=\sqrt{\rho}[P,P^*](\sqrt{\rho}^{-1}\cdot)$. According to~\cite{OPPS2} (Lemma~3.3), the commutator $[\mathscr{L},\mathscr{L}^*]$ is given by the quadratic operator
$$[\mathscr{L},\mathscr{L}^*]=\big\langle (QQ_{\infty}^{-1}Q+2QB^T)D_x,D_x \big\rangle+\frac{1}{4}\big\langle (QQ_{\infty}^{-1}Q+2QB^T)Q_{\infty}^{-1}x,Q_{\infty}^{-1}x\big\rangle,$$
and the operator $\mathscr{L}$ commutes with its $L^2$-adjoint $\mathscr{L}^*$, when acting on $\mathscr{S}(\rr^n)$, 
if and only if the symmetric parts of the two matrices $QQ_{\infty}^{-1}Q+2QB^T$ and $Q_{\infty}^{-1}(QQ_{\infty}^{-1}Q+2QB^T)Q_{\infty}^{-1}$ are zero, that is, when
$$QQ_{\infty}^{-1}Q+QB^T+BQ=0 \quad \textrm{and} \quad Q_{\infty}^{-1}(QQ_{\infty}^{-1}Q+QB^T+BQ)Q_{\infty}^{-1}=0,$$
thanks to the symmetry of the two matrices $Q$ and $Q_{\infty}^{-1}$. These two conditions are equivalent to 
$$QQ_{\infty}^{-1}Q+QB^T+BQ=0.$$
By using the formula (\ref{ai4}) which implies that  
$$0=(Q+BQ_{\infty}+Q_{\infty}B^T)Q_{\infty}^{-1}Q=QQ_{\infty}^{-1}Q+BQ+Q_{\infty}B^TQ_{\infty}^{-1}Q,$$
the previous condition is therefore equivalent to 
\begin{multline*}
0=QQ_{\infty}^{-1}Q+QB^T-QQ_{\infty}^{-1}Q-Q_{\infty}B^TQ_{\infty}^{-1}Q=QB^T-Q_{\infty}B^TQ_{\infty}^{-1}Q=Q_{\infty}[Q_{\infty}^{-1}Q,B^T],
\end{multline*}
i.e. that the following commutator is zero 
$$[Q_{\infty}^{-1}Q,B^T]=0.$$ 
This proves that the operator $P$ commutes with its $L^2_{\mu}$-adjoint~$P^*$, 
\begin{equation}\label{ai5}
\forall u \in \mathscr{S}(\rr^n), \quad [P,P^*](\sqrt{\rho}^{-1}u)=0,
\end{equation} 
when acting on $\sqrt{\rho}^{-1}\mathscr{S}(\rr^n)$,
if and only if 
\begin{equation}\label{ai6}
[Q_{\infty}^{-1}Q,B^T]=0.
\end{equation} 
Setting now
$$m_-=\inf_{(x,\xi) \in \mathcal{S}}\Big\langle\Big(-\frac{1}{2}QQ_{\infty}^{-1}-B\Big)x,\xi\Big\rangle, \quad m_+=\sup_{(x,\xi) \in \mathcal{S}}\Big\langle\Big(-\frac{1}{2}QQ_{\infty}^{-1}-B\Big)x,\xi\Big\rangle,$$
with
$$\mathcal{S}=\Big\{(x,\xi) \in \rr^{2n} :\ \frac{1}{2}|Q^{1/2}\xi|^2+\frac{1}{8}|Q^{1/2}Q_{\infty}^{-1}x|^2=1\Big\},$$
we notice from (\ref{pav6}) that the numerical range is 
\begin{equation}\label{dionk1}
\Sigma(q)=\overline{q(\rr^{2n})}=\big\{z=t(1+i\lambda) \in \cc : \  m_- \leq \lambda \leq m_+, \ t \geq 0\big\}.
\end{equation}
We can compute explicitly the two constants $m_{\pm}$.
By changing $x$ in $-x$, and $\xi$ in $-\xi$, we observe that 
\begin{equation}\label{x2}
m_+=- m_-=\sup_{(x,\xi) \in \mathcal{S}}\Big\langle\Big(\frac{1}{2}QQ_{\infty}^{-1}+B\Big)x,\xi\Big\rangle.
\end{equation}
On the other hand, we notice that 
\begin{align}\label{x1}
m_+=& \sup_{\frac{1}{8}|Q^{\frac{1}{2}}Q_{\infty}^{-1}x|^2 \leq 1}\Big(\sup_{\frac{1}{2}|Q^{\frac{1}{2}}\xi|^2=1-\frac{1}{8}|Q^{\frac{1}{2}}Q_{\infty}^{-1}x|^2}\Big\langle\Big(\frac{1}{2}QQ_{\infty}^{-1}+B\Big)x,\xi\Big\rangle\Big)\\ \notag
=&\sup_{\frac{1}{8}|Q^{\frac{1}{2}}Q_{\infty}^{-1}x|^2 \leq 1}\Big(\sup_{|\xi|^2=2-\frac{1}{4}|Q^{\frac{1}{2}}Q_{\infty}^{-1}x|^2}\Big\langle\Big(\frac{1}{2}Q^{1/2}Q_{\infty}^{-1}+Q^{-1/2}B\Big)x,\xi\Big\rangle\Big)\\ \notag
=&\frac{1}{2}\sup_{|Q^{\frac{1}{2}}Q_{\infty}^{-1}x|^2 \leq 8}\|(Q^{1/2}Q_{\infty}^{-1}+2Q^{-1/2}B)x\|\sqrt{2-\frac{1}{4}|Q^{1/2}Q_{\infty}^{-1}x|^2}.
\end{align}
It follows from (\ref{x1}) that 
\begin{align*}
m_+= & \ \frac{1}{2}\sup_{|y|^2 \leq 8}\|(Q^{1/2}Q_{\infty}^{-1}+2Q^{-1/2}B)Q_{\infty}Q^{-1/2}y\|\sqrt{2-\frac{1}{4}|y|^2}\\
= & \ \frac{1}{2}\sup_{|y|^2 \leq 8}\|(I_n+2Q^{-1/2}BQ_{\infty}Q^{-1/2})y\|\sqrt{2-\frac{1}{4}|y|^2}\\
=& \ \|I_n+2Q^{-1/2}BQ_{\infty}Q^{-1/2}\|\left(\sup_{0 \leq t \leq \sqrt{8}}\frac{t}{2}\sqrt{2-\frac{t^2}{4}}\right)=\|I_n+2Q^{-1/2}BQ_{\infty}Q^{-1/2}\|.
\end{align*}
As the matrix (\ref{skew}) is skew-selfadjoint, we obtain that 
\begin{equation}\label{ai7}
m_{\pm}=\pm\|I_n+2Q^{-1/2}Q_{\infty}B^TQ^{-\frac{1}{2}}\|.
\end{equation}
It follows from (\ref{dionk1}) and (\ref{ai7}) that 
\begin{equation}\label{fr21}
\Sigma(q)=\big\{z=t(1+i\lambda) \in \cc : \  \lambda \in \rr, \  |\lambda| \leq \|I_n+2Q^{-1/2}Q_{\infty}B^TQ^{-\frac{1}{2}}\|, \ t \geq 0\big\}.
\end{equation}
We recover from Proposition~\ref{propo2} and (\ref{bn1}) the result established in~\cite{sector} that the two semigroups $(e^{-t\mathscr{L}})_{t \geq 0}$ and $(e^{tP})_{t \geq 0}$ are analytic in the angular sector
\begin{equation}\label{cx1}
\Sigma_{\theta_2}=\{z \in \cc^* : |\arg z| <\theta_2\} \cup \{0\},
\end{equation}
with 
\begin{equation}\label{cx2}
\textrm{cotan }\theta_2=\|I_n+2Q^{-1/2}Q_{\infty}B^TQ^{-1/2}\|.
\end{equation}
Writing 
\begin{equation}\label{nn7}
\partial\Sigma(q)=e^{i\theta_-}\rr_+ \cup e^{i\theta_+}\rr_+,
\end{equation}
we have indeed that 
\begin{equation}\label{nn8}
0 \leq \theta_+=-\theta_- <\frac{\pi}{2}
\end{equation}
and
\begin{equation}\label{nn5}
0<\delta=\inf\Big(\frac{\pi}{2}-\theta_+, \theta_-+\frac{\pi}{2}\Big)=\frac{\pi}{2}-\theta_+=\theta_2\leq \frac{\pi}{2}.
\end{equation}

Let us now explain why this angle of the analyticity sector $\theta_2$ is actually sharp and how it relates with the spectral angle of the Ornstein-Uhlenbeck operator $P$ given by the spectral angle of the matrix $B$, 
$$0 \leq \theta_{\textrm{spec}}:=\textrm{arctan}\Big(\sup_{z \in \sigma(B)}\Big|\frac{\textrm{Im }z}{\textrm{Re }z}\Big|\Big) < \frac{\pi}{2}.$$
To that end, we study under which condition the quadratic operator $\mathscr{L}$ admits some eigenvalues that are regularly spaced out along an half-line parallel to one of the half-lines composing the boundary of the numerical range $\partial \Sigma(q)$. 
Let us consider the two quadratic forms obtained after rotations preserving the non-negativity, but not the positive definiteness of their real parts
$$q_1(x,\xi)=(\|I_n+2Q^{-1/2}Q_{\infty}B^TQ^{-1/2}\|+i)q(x,\xi)$$
and
$$q_2(x,\xi)=(\|I_n+2Q^{-1/2}Q_{\infty}B^TQ^{-1/2}\|-i)q(x,\xi),$$
i.e.
$$\forall j =1,2, \quad \textrm{Re }q_j \geq 0, \quad  \textrm{Re }q_j \not\gg 0.$$
These two quadratic forms are obviously elliptic on the phase space $\rr_x^{n} \times \rr_{\xi}^n$, as the quadratic form $q$ is.
If $j \in \{1,2\}$, the Hamilton map $F_j$ of the quadratic form $q_j$ is given by  
$$F_j=\big(\|I_n+2Q^{-1/2}Q_{\infty}B^TQ^{-1/2}\|+i(-1)^{j+1}\big)F.$$
It implies that 
\begin{equation}\label{a1}
\textrm{Re }F_j=\|I_n+2Q^{-1/2}Q_{\infty}B^TQ^{-1/2}\|\textrm{Re }F+(-1)^{j}\textrm{Im }F,
\end{equation}
\begin{equation}\label{a2}
\textrm{Im }F_j=\|I_n+2Q^{-1/2}Q_{\infty}B^TQ^{-1/2}\|\textrm{Im }F+(-1)^{j+1}\textrm{Re }F.
\end{equation}
For $X \in \rr^{2n}$, the non-negativity of the quadratic form $\textrm{Re }q_j \geq 0$ implies that 
$$\textrm{Re }q_j(X)=\sigma\big(X,(\textrm{Re }F_j)X\big)=0 \Longleftrightarrow (\textrm{Re }F_j)X=0.$$
Indeed, if $\textrm{Re }q_j(X)=0$, the Cauchy-Schwarz inequality ensures that 
$$\forall Y \in \rr^{2n}, \quad (\textrm{Re }q_j)(Y;X)^2= \sigma\big(Y,(\textrm{Re }F_j)X\big)^2 \leq [\textrm{Re }q_j(Y)][\textrm{Re }q_j(X)]=0,$$
with $\textrm{Re }q_j(Y;X)$ denoting the polarized form associated to $\textrm{Re }q_j$. The non-degeneracy of the symplectic form then implies that $(\textrm{Re }F_j)X=0$.
As the quadratic form $\textrm{Re }q \gg 0$ is positive definite, the above characterization allows to recover that the Hamilton map $\textrm{Re }F$ is invertible.
We deduce from (\ref{a1}) that for all $X=(x,\xi) \in \rr^{2n}$,
\begin{multline}\label{a0}
\textrm{Re }q_j(X)=0  \Leftrightarrow (\textrm{Re }F_j)X=0\\ \Leftrightarrow (\textrm{Re }F)^{-1}(\textrm{Im }F)X=(-1)^{j+1}\|I_n+2Q^{-1/2}Q_{\infty}B^TQ^{-1/2}\|X.
\end{multline}
We observe from (\ref{F}) that
\begin{equation}\label{pav22}
\textrm{Im }F=\frac{1}{4}\lv\begin{array}{cc}
 -(QQ_{\infty}^{-1}+2B)&0 \\
0  & (QQ_{\infty}^{-1}+2B)^T       
\end{array}\rv.
\end{equation}
It follows from (\ref{ai9}) and (\ref{pav22}) that 
\begin{equation}\label{a3}
(\textrm{Re }F)^{-1}\textrm{Im }F=\lv\begin{array}{cc}
0 & -2Q_{\infty}Q^{-1}Q_{\infty}(Q_{\infty}^{-1}Q+2B^T)\\
-\frac{1}{2}Q^{-1}(QQ_{\infty}^{-1}+2B)    & 0     
\end{array}\rv.
\end{equation}
We notice that 
\begin{equation}\label{a4}
-2Q_{\infty}Q^{-1}Q_{\infty}(Q_{\infty}^{-1}Q+2B^T)=-2Q_{\infty}Q^{-1/2}(I_n+2Q^{-1/2}Q_{\infty}B^TQ^{-1/2})Q^{1/2},
\end{equation}
\begin{equation}\label{a5}
-\frac{1}{2}Q^{-1}(QQ_{\infty}^{-1}+2B)=-\frac{1}{2}Q^{-1/2}(I_n+2Q^{-1/2}BQ_{\infty}Q^{-1/2})Q^{1/2}Q_{\infty}^{-1}.
\end{equation}
As the matrix
$$A=I_n+2Q^{-1/2}Q_{\infty}B^TQ^{-1/2},$$
is skew-adjoint,  we deduce from (\ref{a0}), (\ref{a3}) (\ref{a4}) and (\ref{a5}) that for all $(x,\xi) \in \rr^{2n}$, we have $(x,\xi) \in \textrm{Ker}(\textrm{Re }F_j)$ if and only if 
\begin{equation}\label{a6} 
\left\{
\begin{array}{cl}
AQ^{1/2}\xi&=\frac{1}{2}(-1)^{j}\|A\|  Q^{1/2}Q_{\infty}^{-1}x, \\
AQ^{1/2}Q_{\infty}^{-1}x&=2(-1)^{j+1}\|A\|Q^{1/2}\xi.
\end{array}
\right.
\end{equation}
Setting 
$$D=Q^{1/2}Q_{\infty}^{-1}Q^{1/2},$$
we can rewrite (\ref{a6}) as 
\begin{equation}\label{a6.4} 
\left\{
\begin{array}{cl}
AQ^{1/2}\xi&=\frac{1}{2}(-1)^{j}\|A\|  DQ^{-1/2}x, \\
ADQ^{-1/2}x&=2(-1)^{j+1}\|A\|Q^{1/2}\xi.
\end{array}
\right.
\end{equation}
We use the following instrumental lemma whose proof is given in Appendix (Section~\ref{appendix}): 

\bigskip

\begin{lemma}\label{lem1.1}
If $j \in \{1,2\}$, the singular space $S_j$ of the quadratic form $q_j$ satisfies 
\begin{multline*}
S_j=\Big(\bigcap_{k=0}^{2n-1}\emph{\textrm{Ker}}\big(\emph{\textrm{Re }}F_j(\emph{\textrm{Im }}F_j)^k\big)\Big)\cap \rr^{2n} \neq \{0\} \\ \Rightarrow \Big( \bigcap_{k=0}^{n-1}\emph{\textrm{Ker}}
\big((A^2+\|A\|^2)D^{k}\big)\cap \emph{\textrm{Ker}}
\big([A,D^k]\big)\Big)\cap \rr^{n}\neq\{0\}.
\end{multline*}
If $j \in \{1,2\}$ and $A \neq 0$,  the singular space $S_j$ of the quadratic form $q_j$ satisfies 
\begin{multline*}
\Big( \bigcap_{k=0}^{n-1}\emph{\textrm{Ker}}
\big((A^2+\|A\|^2)D^{k}\big)\cap \emph{\textrm{Ker}}
\big([A,D^k]\big)\Big)\cap \rr^{n}\neq\{0\} \\ \Rightarrow S_j=\Big(\bigcap_{k=0}^{2n-1}\emph{\textrm{Ker}}\big(\emph{\textrm{Re }}F_j(\emph{\textrm{Im }}F_j)^k\big)\Big)\cap \rr^{2n} \neq \{0\} .
\end{multline*}
\end{lemma}

\bigskip

Let assume for now that $A \neq 0$. According to (\ref{fr21}), the angular sector $\Sigma(q)$ has a positive angle
\begin{equation}\label{nn6}
\Sigma(q)=\big\{z=t(1+i\lambda) \in \cc : \  \lambda \in \rr, \  |\lambda| \leq \|A\|, \ t \geq 0\big\},
\end{equation}
since $\|A\|>0$.
We can deduce from Proposition~\ref{prop4} and Lemma~\ref{lem1.1} that the spectra of the two quadratic operators $q_1^w(x,D_x)$ and  $q_2^w(x,D_x)$ both contain some eigenvalues regularly spaced out along an half-line parallel to the imaginary axis in the complex plane if and only if 
\begin{equation}\label{fr20}
\Big( \bigcap_{k=0}^{n-1}\textrm{Ker}\big((A^2+\|A\|^2)D^{k}\big)\cap \textrm{Ker}\big([A,D^k]\big)\Big)\cap \rr^{n}\neq\{0\}.
\end{equation}
When the condition (\ref{fr20}) holds, the spectrum of the quadratic operator $q^w(x,D_x)$ thus contains some eigenvalues regularly spaced out along two half-lines parallel to the half-lines $D_{\pm}$ composing the boundary of the numerical range $\partial \Sigma(q)=D_+ \cup D_-$, with one parallel to $D_+$, the other parallel to $D_-$,
$$D_{\pm}=\big\{z=t(1\pm i\|A\|) \in \cc : \ t \geq 0\big\}.$$
According to the characterization (\ref{dani1}), the localization of these eigenvalues prevents the semigroup $(e^{-t\mathscr{L}})_{t \geq 0}$ to be analytic in the angular sector $\Sigma_{\delta'}$
$$\Sigma_{\delta'}=\{z \in \cc^* : |\arg z| <\delta'\} \cup \{0\},$$
with $\theta_2 <\delta'$, since $\theta_{\pm}=\pm (\frac{\pi}{2}-\theta_2)$, and
$$\partial\Sigma(q)=e^{i\theta_-}\rr_+ \cup e^{i\theta_+}\rr_+=e^{-i(\frac{\pi}{2}-\theta_2)}\rr_+ \cup e^{i(\frac{\pi}{2}-\theta_2)}\rr_+.$$ 
According to (\ref{bn1}), it implies that the semigroup $(T(t))_{t \geq 0}=(e^{tP})_{t \geq 0}$ cannot be analytic in the angular sector $\Sigma_{\delta'}$ with $\theta_2 <\delta'$. We recover here that $\theta_2$ the angle of the analyticity sector of the semigroup generated by the Ornstein-Uhlenbeck operator $P$ is sharp. In this case, we also notice that the spectral angle of the quadratic operator $\mathscr{L}$, which is from (\ref{cx3.9}) the same as the spectral angle of the operator $P$ is then exactly equal to 
$$\theta_{\textrm{spec}}=\theta_+=\frac{\pi}{2}-\theta_2,$$
since there are some eigenvalues regularly spaced out along two half-lines parallel to the half-lines $D_{\pm}$ composing the boundary of the numerical range $\partial \Sigma(q)=D_+ \cup D_-$. In this case, the spectral angle of the operator $P$ in $L^2_{\mu}$ and the angle of the analyticity sector of its associated semigroup $(T(t))_{t \geq 0}$ are complementary
$$\theta_2 +\theta_{\textrm{spec}}= \frac{\pi}{2}.$$
If $A=0$, it follows from (\ref{15}) and (\ref{nn6}) that the numerical range $\Sigma(q)$ containing the spectrum of the quadratic operator $\mathscr{L}$ is equal to $\Sigma(q)=\rr_+$. The spectral angle of the operator $\mathscr{L}$, as well as the one of the Ornstein-Uhlenbeck operator $P$ is then zero $\theta_{\textrm{spec}}=0$.
According to (\ref{cx1}) and (\ref{cx2}), the angle of the analyticity sector is then maximal and equal to $\frac{\pi}{2}$, since
$$\textrm{cotan }\theta_2=\|I_n+2Q^{-1/2}Q_{\infty}B^TQ^{-1/2}\|=\|A\|=0.$$
When $A=0$, the spectral angle of the operator $P$ in $L^2_{\mu}$ and the angle of the analyticity sector of its associated semigroup $(T(t))_{t \geq 0}$ are as well complementary.
As the condition 
$$\Big( \bigcap_{k=0}^{n-1}\textrm{Ker}\big((A^2+\|A\|^2)D^{k}\big)\cap \textrm{Ker}\big([A,D^k]\big)\Big)\cap \rr^{n}=\rr^{n}\neq\{0\},$$
trivially holds when $A=0$, we deduce that the condition (\ref{fr20}) always implies that the spectral angle of the operator $P$ in $L^2_{\mu}$ and the angle of the analyticity sector of its associated semigroup $(T(t))_{t \geq 0}$ are complementary.

Let us now consider the case when 
\begin{equation}\label{cx4}
\Big( \bigcap_{k=0}^{n-1}\textrm{Ker}\big((A^2+\|A\|^2)D^{k}\big)\cap \textrm{Ker}\big([A,D^k]\big)\Big)\cap \rr^{n}=\{0\}.
\end{equation}
We then deduce from Proposition~\ref{prop4} and Lemma~\ref{lem1.1} that the spectra of the two quadratic operators $q_1^w(x,D_x)$ and  $q_2^w(x,D_x)$ do not contain some eigenvalues regularly spaced out along any half-line parallel to the imaginary axis. According to the structure of the spectrum of the quadratic operator $\mathscr{L}$ 
given in (\ref{15}), it implies from (\ref{nn7}), (\ref{nn8}) and (\ref{nn5}) that the spectral angle of the operator $\mathscr{L}$ is strictly smaller
$$0 \leq \theta_{\textrm{spec}}< \theta_+=\frac{\pi}{2}-\theta_2,$$
than the angle defining the numerical range
$$\partial\Sigma(q)=e^{i\theta_-}\rr_+ \cup e^{i\theta_+}\rr_+=e^{-i(\frac{\pi}{2}-\theta_2)}\rr_+ \cup e^{i(\frac{\pi}{2}-\theta_2)}\rr_+.$$
In this case, the spectral angle of the operator $P$ in $L^2_{\mu}$ and the angle of the analyticity sector of its associated semigroup $(T(t))_{t \geq 0}$ fail to be complementary.
As highlighted in Example~1 (Section~\ref{main}), the condition (\ref{cx4}) can only hold in the case of a non-normal Ornstein-Ulhenbeck operator when acting on $\sqrt{\rho}^{-1}\mathscr{S}(\rr^n)$, that is, when 
$$[Q_{\infty}^{-1}Q,B^T]\neq 0.$$ 
If $[Q_{\infty}^{-1}Q,B^T]\neq 0$, the pseudospectral phenomena uncovered in Theorem~2.4 in~\cite{OPPS2}, namely the fact that the resolvent of the Ornstein-Uhlenbeck operator $P$ exhibits rapid growth in norm along all half-lines starting from $-\frac{1}{2}\textrm{Tr}(B)$ and contained in the open angular sector 
$$-\frac{1}{2}\textrm{Tr}(B)-\stackrel{\circ}{\Sigma(q)}=\Big\{-\frac{1}{2}\textrm{Tr}(B)-t(1+i\lambda) \in \cc^* : \  |\lambda| < \|A\|, \ t>0\Big\},$$ 
that is
$$\forall |\lambda| < \|A\|, \forall N \in \mathbb{N}, \quad \lim_{t \to +\infty}\frac{\big\|\big(P+\frac{1}{2}\textrm{Tr}(B)+t(1+i\lambda)\big)^{-1}\big\|_{\mathcal{L}(L^2_{\mu})}}{t^N}=+\infty,$$
actually prevent the angle of the analyticity sector of the semigroup to be larger than $\theta_2$, even if the spectrum of $P$ lies more deeply inside the angular sector $-\frac{1}{2}\textrm{Tr}(B)-\Sigma(q)$, as the spectral angle satisfies
$$0 \leq \theta_{\textrm{spec}}< \theta_+=\frac{\pi}{2}-\theta_2.$$
This ends the proof of Theorem~\ref{th}.

\section{Appendix}\label{appendix}

This appendix is devoted to the proof of Lemma~\ref{lem1.1}. The notations used here are the ones introduced in Section~\ref{uhl}.

We first establish by recurrence on the non-negative integer $m \in \nn$ that
$$(x,\xi) \in \Big(\bigcap_{k=0}^m\textrm{Ker}\big(\textrm{Re }F_j(\textrm{Im }F_j)^k\big) \Big)\cap \rr^{2n},$$
implies that for all $0 \leq k \leq m$,
\begin{equation}\label{vn1}
\left\{
\begin{array}{rl}
AD^kQ^{1/2}\xi=& \frac{1}{2}(-1)^{j}\|A\|D^{k+1}Q^{-1/2}x,\\
AD^{k+1}Q^{-1/2}x=& 2(-1)^{j+1}\|A\|D^kQ^{1/2}\xi,\\
(A^2+\|A\|^2)D^{k+1}Q^{-1/2}x=& 0,\\
(A^2+\|A\|^2)D^k Q^{1/2}\xi=& 0,\\
\textrm{[}A,D^k\textrm{]}DQ^{-1/2}x=& 0,\\
\textrm{[}A,D^k\textrm{]}Q^{1/2}\xi=& 0.\\
\end{array}
\right.
\end{equation}
When $m=0$, the two first equalities in (\ref{vn1}) are given by (\ref{a6.4}), whereas the last two are trivial as the commutator $[A,I]=0$ is zero. On the other hand, we deduce from (\ref{a6.4}) that 
$$A^2Q^{1/2}\xi=\frac{1}{2}(-1)^{j}\|A\|  ADQ^{-1/2}x=-\|A\|^2Q^{1/2}\xi$$
and
$$A^2DQ^{-1/2}x=2(-1)^{j+1}\|A\|AQ^{1/2}\xi=-\|A\|^2DQ^{-1/2}x.$$
It proves that the third and fourth equalities in (\ref{vn1}) hold as well in this case.
Let us now assume that the above implication holds at the rank $m$ and let us consider 
$$(x,\xi) \in \Big(\bigcap_{k=0}^{m+1}\textrm{Ker}\big(\textrm{Re }F_j(\textrm{Im }F_j)^k\big)\Big) \cap \rr^{2n}.$$
We have in particular that 
$$(x,\xi) \in \Big(\bigcap_{k=0}^{m}\textrm{Ker}\big(\textrm{Re }F_j(\textrm{Im }F_j)^k\big)\Big)\cap \rr^{2n}$$ 
and 
$$(\tilde{x},\tilde{\xi})=(\textrm{Im }F_j)(x,\xi) \in \Big(\bigcap_{k=0}^{m}\textrm{Ker}\big(\textrm{Re }F_j(\textrm{Im }F_j)^k\big)\Big)\cap \rr^{2n}.$$
With $X=(x,\xi)$ and $\tilde{X}=(\tilde{x},\tilde{\xi})$, it follows from (\ref{nn1}), (\ref{a2}) and (\ref{a0}) that 
\begin{multline*}
(\textrm{Im }F_j)X=(\textrm{Re }F)(\|A\|(\textrm{Re }F)^{-1}(\textrm{Im }F)X+(-1)^{j+1}X)\\
=(-1)^{j+1}(1+\|A\|^2)(\textrm{Re }F)X,
\end{multline*}
that is
\begin{equation}\label{nn2}
\left\{
\begin{array}{rl}
\tilde{x}=&\frac{1}{2}(-1)^{j+1}(1+\|A\|^2)Q\xi\\
\tilde{\xi}=&\frac{1}{8}(-1)^{j}(1+\|A\|^2)Q_{\infty}^{-1}QQ_{\infty}^{-1}x=\frac{1}{8}(-1)^{j}(1+\|A\|^2)Q^{-1/2}D^2Q^{-1/2}x
\end{array}
\right.
\end{equation}
By applying the first equation in (\ref{vn1}) for $(\tilde{x},\tilde{\xi})$ at the rank $m$, 
$$AD^mQ^{1/2}\tilde{\xi}=\frac{1}{2}(-1)^{j}\|A\|D^{m+1}Q^{-1/2}\tilde{x},$$
we obtain from (\ref{nn2}) that 
$$AD^mQ^{1/2}Q^{-1/2}D^2Q^{-1/2}x=2(-1)^{j+1}\|A\|D^{m+1}Q^{-1/2}Q\xi,$$
i.e. the second equation in (\ref{vn1}) for $(x,\xi)$ at the rank $m+1$,
$$AD^{m+2}Q^{-1/2}x=2(-1)^{j+1}\|A\|D^{m+1}Q^{1/2}\xi.$$
By applying the second equation in (\ref{vn1}) for $(\tilde{x},\tilde{\xi})$ at the rank $m$, 
$$AD^{m+1}Q^{-1/2}\tilde{x}= 2(-1)^{j+1}\|A\|D^mQ^{1/2}\tilde{\xi},$$
we obtain from (\ref{nn2}) that 
$$AD^{m+1}Q^{-1/2}Q\xi= \frac{1}{2}(-1)^{j}\|A\|D^mQ^{1/2}Q^{-1/2}D^2Q^{-1/2}x,$$
i.e. the first equation in (\ref{vn1}) for $(x,\xi)$ at the rank $m+1$,
$$AD^{m+1}Q^{1/2}\xi= \frac{1}{2}(-1)^{j}\|A\|D^{m+2}Q^{-1/2}x.$$
By applying the third equation in (\ref{vn1}) for $(\tilde{x},\tilde{\xi})$ at the rank $m$, 
$$(A^2+\|A\|^2)D^{m+1}Q^{-1/2}\tilde{x}= 0,$$
we obtain from (\ref{nn2}) that 
$$(A^2+\|A\|^2)D^{m+1}Q^{-1/2}Q\xi= 0,$$
i.e. the fourth equation in (\ref{vn1}) for $(x,\xi)$ at the rank $m+1$,
$$(A^2+\|A\|^2)D^{m+1}Q^{1/2}\xi= 0.$$
By applying the fourth equation in (\ref{vn1}) for $(\tilde{x},\tilde{\xi})$ at the rank $m$, 
$$(A^2+\|A\|^2)D^m Q^{1/2}\tilde{\xi}= 0,$$
we obtain from (\ref{nn2}) that 
$$(A^2+\|A\|^2)D^m Q^{1/2}Q^{-1/2}D^2Q^{-1/2}x= 0,$$
i.e. the third equation in (\ref{vn1}) for $(x,\xi)$ at the rank $m+1$,
$$(A^2+\|A\|^2)D^{m+2}Q^{-1/2}x= 0.$$
Next, it follows from the second equation in (\ref{vn1}) for $(x,\xi)$ at the rank $0$ and $m+1$ that 
\begin{multline*}
[A,D^{m+1}]DQ^{-1/2}x=(AD^{m+2}Q^{-1/2}x)-D^{m+1}(ADQ^{-1/2}x)\\
=2(-1)^{j+1}\|A\|(D^{m+1}Q^{1/2}\xi-D^{m+1}Q^{1/2}\xi)=0,
\end{multline*}
i.e the fifth equation in (\ref{vn1}) for $(x,\xi)$ at the rank $m+1$.
Then, it follows from the first equation in (\ref{vn1}) for $(x,\xi)$ at the rank $0$ and $m+1$ that 
\begin{multline*}
[A,D^{m+1}]Q^{1/2}\xi=(AD^{m+1}Q^{1/2}\xi)-D^{m+1}(AQ^{1/2}\xi)\\
=\frac{1}{2}(-1)^{j}\|A\|(D^{m+2}Q^{-1/2}x-D^{m+1}DQ^{-1/2}x)=0,
\end{multline*}
i.e the last equation in (\ref{vn1}) for $(x,\xi)$ at the rank $m+1$.

By using next that the matrices $D$ and $Q$ are invertible, we deduce from (\ref{vn1}) that for all $m \in \nn$,
\begin{multline}\label{fr14}
\Big(\bigcap_{k=0}^m\textrm{Ker}\big(\textrm{Re }F_j(\textrm{Im }F_j)^k\big)\Big)\cap \rr^{2n} \neq \{0\}\\ \Rightarrow \Big( \bigcap_{k=0}^{m}\textrm{Ker}
\big((A^2+\|A\|^2)D^{k}\big)\cap \textrm{Ker}
\big([A,D^k]\big)\Big)\cap \rr^{n}\neq\{0\}.
\end{multline}
Let us assume now that $A \neq 0$ and that 
$$\Big( \bigcap_{k=0}^{m}\textrm{Ker}\big((A^2+\|A\|^2)D^{k}\big)\cap \textrm{Ker}\big([A,D^k]\big)\Big)\cap \rr^{n}\neq\{0\},$$
with $m \in \nn$. Let 
$$\zeta_0 \in \Big( \bigcap_{k=0}^{m}\textrm{Ker}\big((A^2+\|A\|^2)D^{k}\big)\cap \textrm{Ker}\big([A,D^k]\big)\Big)\cap \rr^{n} \setminus \{0\}.$$
We define
\begin{equation}\label{bn8}
x_0=\frac{2}{ \|A\|}(-1)^{j}Q^{1/2}D^{-1}A\zeta_0 \quad \textrm{ and }\quad  \xi_0=Q^{-1/2}\zeta_0 \neq 0.
\end{equation}
We first have $(x_0,\xi_0) \in \rr^{2n} \setminus \{0\}$. It follows from (\ref{a6}), (\ref{a6.4}) and (\ref{bn8}) that $(x_0,\xi_0) \in \textrm{Ker}(\textrm{Re }F_j)$ since 
\begin{equation}\label{rf1}
AQ^{1/2}\xi_0=A\zeta_0=\frac{1}{2}(-1)^{j}\|A\|  DQ^{-1/2}x_0
\end{equation}
and
\begin{multline}\label{rf2}
ADQ^{-1/2}x_0=\frac{2}{ \|A\|}(-1)^{j}ADQ^{-1/2}Q^{1/2}D^{-1}A\zeta_0\\
=\frac{2}{ \|A\|}(-1)^{j}A^2\zeta_0=2(-1)^{j+1} \|A\|\zeta_0
=2(-1)^{j+1}\|A\|Q^{1/2}\xi_0,
\end{multline}
as $A^2\zeta_0=-\|A\|^2\zeta_0$.
We also notice from (\ref{bn8}) that 
\begin{equation}\label{fr1}
\forall 0 \leq k \leq m, \quad [A,D^k]Q^{1/2}\xi_0=[A,D^k]\zeta_0=0
\end{equation}
and
\begin{multline*}
\forall 0 \leq k \leq m, \quad [A,D^k]DQ^{-1/2}x_0=\frac{2}{ \|A\|}(-1)^{j}[A,D^k]DQ^{-1/2}Q^{1/2}D^{-1}A\zeta_0\\
=\frac{2}{ \|A\|}(-1)^{j}(AD^kA\zeta_0-D^kA^2\zeta_0)=\frac{2}{ \|A\|}(-1)^{j}(A^2D^k\zeta_0-D^kA^2\zeta_0),
\end{multline*}
since $[A,D^k]\zeta_0=0$. By using that 
$$A^2D^k\zeta_0=-\|A\|^2D^k\zeta_0 \quad \textrm{and} \quad A^2\zeta_0=-\|A\|^2\zeta_0,$$
we obtain that 
\begin{equation}\label{fr2}
\forall 0 \leq k \leq m, \quad [A,D^k]DQ^{-1/2}x_0=2(-1)^{j+1}\|A\|(D^k\zeta_0-D^k\zeta_0)=0.
\end{equation}
By using that $X_0=(x_0,\xi_0) \in \textrm{Ker}(\textrm{Re }F_j)$, it follows from (\ref{nn1}), (\ref{a2}) and (\ref{a0}) that 
\begin{align*}
\tilde{X}_0=(\textrm{Im }F_j)X_0=& \ (\textrm{Re }F)(\|A\|(\textrm{Re }F)^{-1}(\textrm{Im }F)X_0+(-1)^{j+1}X_0) \\
=&\ (-1)^{j+1}(1+\|A\|^2)(\textrm{Re }F)X_0\\ =& \ \lv\begin{array}{cc}
 0& \frac{1}{2}(-1)^{j+1}(1+\|A\|^2)Q \\
  \frac{1}{8}(-1)^{j}(1+\|A\|^2)Q^{-1/2}D^2Q^{-1/2}  & 0    
\end{array}\rv X_0.
\end{align*}
By using again that $X_0=(x_0,\xi_0) \in \textrm{Ker}(\textrm{Re }F_j)$, it follows from (\ref{a6}) and (\ref{a6.4}) that 
$$\tilde{X}_0=(\textrm{Im }F_j)X_0=\frac{1+\|A\|^2}{4\|A\|} \lv\begin{array}{cc}
 Q^{1/2}ADQ^{-1/2}& 0 \\
 0 & Q^{-1/2}DAQ^{1/2}    
\end{array}\rv X_0.$$
In order to establish that 
$$X_0 \in \Big(\bigcap_{k=0}^m\textrm{Ker}\big(\textrm{Re }F_j(\textrm{Im }F_j)^k\big)\Big)\cap \rr^{2n},$$
it is therefore sufficient to check that 
$$\forall 0 \leq k \leq m, \quad (Q^{1/2}(AD)^kQ^{-1/2}x_0,Q^{-1/2}(DA)^kQ^{1/2}\xi_0) \in \textrm{Ker}(\textrm{Re }F_j).$$
According to (\ref{a6}) and (\ref{a6.4}), it is sufficient to check that for all $0 \leq k \leq m$,
$$AQ^{1/2}Q^{-1/2}(DA)^kQ^{1/2}\xi_0=\frac{1}{2}(-1)^{j}\|A\|  DQ^{-1/2}Q^{1/2}(AD)^kQ^{-1/2}x_0$$
and
$$ADQ^{-1/2}Q^{1/2}(AD)^kQ^{-1/2}x_0=2(-1)^{j+1}\|A\|Q^{1/2}Q^{-1/2}(DA)^kQ^{1/2}\xi_0,$$
that is, for all $0 \leq k \leq m$,
$$A(DA)^kQ^{1/2}\xi_0=\frac{1}{2}(-1)^{j}\|A\|  D(AD)^kQ^{-1/2}x_0=\frac{1}{2}(-1)^{j}\|A\|  (DA)^kDQ^{-1/2}x_0$$
and
$$(AD)^{k+1}Q^{-1/2}x_0=2(-1)^{j+1}\|A\|(DA)^kQ^{1/2}\xi_0.$$
It follows from (\ref{rf1}) and (\ref{rf2}) that for all $0 \leq k \leq m$,
$$A(DA)^kQ^{1/2}\xi_0=(AD)^kAQ^{1/2}\xi_0=\frac{1}{2}(-1)^{j}\|A\| (AD)^kDQ^{-1/2}x_0$$
and
$$(AD)^{k+1}Q^{-1/2}x_0=(AD)^{k}ADQ^{-1/2}x_0=2(-1)^{j+1}\|A\|(AD)^{k}Q^{1/2}\xi_0.$$
It is therefore sufficient to check that for all $0 \leq k \leq m$,
$$(AD)^kDQ^{-1/2}x_0=(DA)^kDQ^{-1/2}x_0$$
and
$$(AD)^{k}Q^{1/2}\xi_0=(DA)^kQ^{1/2}\xi_0.$$
To that end, we first establish by recurrence that 
\begin{equation}\label{fr5}
\forall 0 \leq 2k \leq m, \quad (AD)^{2k}Q^{1/2}\xi_0=(DA)^{2k}Q^{1/2}\xi_0=(-1)^k\|A\|^{2k}D^{2k}Q^{1/2}\xi_0,
\end{equation}
\begin{multline}\label{fr6}
\forall 0 \leq 2k+1 \leq m, \quad (AD)^{2k+1}Q^{1/2}\xi_0=(DA)^{2k+1}Q^{1/2}\xi_0\\
=\frac{1}{2}(-1)^{j+k}\|A\|^{2k+1}D^{2k+1}DQ^{-1/2}x_0.
\end{multline}
The formula (\ref{fr5}) trivially holds true when $2k=0$, whereas the formula (\ref{fr6}) holds true as well when $2k+1=1$, 
$$(AD)Q^{1/2}\xi_0=(DA)Q^{1/2}\xi_0=\frac{1}{2}(-1)^{j}\|A\|DDQ^{-1/2}x_0,$$
since $[A,D]Q^{1/2}\xi_0=0$ and $AQ^{1/2}\xi_0=\frac{1}{2}(-1)^{j}\|A\|DQ^{-1/2}x_0$ thanks to (\ref{rf1}) and (\ref{fr1}).
If we assume that (\ref{fr5}) holds for $0 \leq 2k \leq m$ and that $0 \leq 2k+1 \leq m$, we deduce from (\ref{rf1}) and (\ref{fr1}) that 
\begin{multline*}
(AD)^{2k+1}Q^{1/2}\xi_0=(AD)(AD)^{2k}Q^{1/2}\xi_0=(-1)^k\|A\|^{2k}(AD)D^{2k}Q^{1/2}\xi_0\\
=(-1)^k\|A\|^{2k}D^{2k+1}AQ^{1/2}\xi_0=\frac{1}{2}(-1)^{k+j}\|A\|^{2k+1}D^{2k+1}DQ^{-1/2}x_0
\end{multline*}
and 
\begin{multline*}
(DA)^{2k+1}Q^{1/2}\xi_0=(DA)(DA)^{2k}Q^{1/2}\xi_0=(-1)^k\|A\|^{2k}(DA)D^{2k}Q^{1/2}\xi_0\\
=(-1)^k\|A\|^{2k}D^{2k+1}AQ^{1/2}\xi_0=\frac{1}{2}(-1)^{k+j}\|A\|^{2k+1}D^{2k+1}DQ^{-1/2}x_0.
\end{multline*}
If we assume that (\ref{fr6}) holds for $0 \leq 2k+1 \leq m$ and that $0 \leq 2k+2 \leq m$, we deduce from (\ref{rf2}) and (\ref{fr2}) that 
\begin{multline*}
(AD)^{2k+2}Q^{1/2}\xi_0=(AD)(AD)^{2k+1}Q^{1/2}\xi_0=\frac{1}{2}(-1)^{j+k}\|A\|^{2k+1}AD^{2k+2}DQ^{-1/2}x_0\\
=\frac{1}{2}(-1)^{j+k}\|A\|^{2k+1}D^{2k+2}ADQ^{-1/2}x_0=(-1)^{k+1}\|A\|^{2k+2}D^{2k+2}Q^{1/2}\xi_0
\end{multline*}
and 
\begin{multline*}
(DA)^{2k+2}Q^{1/2}\xi_0=(DA)(DA)^{2k+1}Q^{1/2}\xi_0=\frac{1}{2}(-1)^{j+k}\|A\|^{2k+1}(DA)D^{2k+1}DQ^{-1/2}x_0\\
=\frac{1}{2}(-1)^{j+k}\|A\|^{2k+1}D^{2k+2}ADQ^{-1/2}x_0=(-1)^{k+1}\|A\|^{2k+2}D^{2k+2}Q^{1/2}\xi_0.
\end{multline*}
Secondly, we establish that 
\begin{multline}\label{fr7}
\forall 0 \leq 2k \leq m, \quad (AD)^{2k}DQ^{-1/2}x_0=(DA)^{2k}DQ^{-1/2}x_0\\
=(-1)^k\|A\|^{2k}D^{2k}DQ^{-1/2}x_0,
\end{multline}
\begin{multline}\label{fr8}
\forall 0 \leq 2k+1 \leq m, \quad (AD)^{2k+1}DQ^{-1/2}x_0=(DA)^{2k+1}DQ^{-1/2}x_0\\
=2(-1)^{j+k+1}\|A\|^{2k+1}D^{2k+1}Q^{1/2}\xi_0.
\end{multline}
The formula (\ref{fr7}) trivially holds true when $2k=0$, whereas the formula (\ref{fr8}) holds true as well when $2k+1=1$, 
$$(AD)DQ^{-1/2}x_0=(DA)DQ^{-1/2}x_0=2(-1)^{j+1}\|A\|DQ^{1/2}\xi_0,$$
since $[A,D]DQ^{-1/2}x_0=0$ and $ADQ^{-1/2}x_0=2(-1)^{j+1}\|A\|Q^{1/2}\xi_0$ thanks to (\ref{rf2}) and (\ref{fr2}).
If we assume that (\ref{fr7}) holds for $0 \leq 2k \leq m$ and that $0 \leq 2k+1 \leq m$, we deduce from (\ref{rf2}) and (\ref{fr2}) that 
\begin{multline*}
(AD)^{2k+1}DQ^{-1/2}x_0=(AD)(AD)^{2k}DQ^{-1/2}x_0=(-1)^k\|A\|^{2k}AD^{2k+1}DQ^{-1/2}x_0\\
=(-1)^k\|A\|^{2k}D^{2k+1}ADQ^{-1/2}x_0=2(-1)^{j+k+1}\|A\|^{2k+1}D^{2k+1}Q^{1/2}\xi_0
\end{multline*}
and 
\begin{multline*}
(DA)^{2k+1}DQ^{-1/2}x_0=(DA)(DA)^{2k}DQ^{-1/2}x_0=(-1)^k\|A\|^{2k}(DA)D^{2k}DQ^{-1/2}x_0\\
=(-1)^k\|A\|^{2k}D^{2k+1}ADQ^{-1/2}x_0=2(-1)^{j+k+1}\|A\|^{2k+1}D^{2k+1}Q^{1/2}\xi_0.
\end{multline*}
If we assume that (\ref{fr8}) holds for $0 \leq 2k+1 \leq m$ and that $0 \leq 2k+2 \leq m$, we deduce from (\ref{rf1}) and (\ref{fr1}) that 
\begin{align*}
& \ (AD)^{2k+2}DQ^{-1/2}x_0=(AD)(AD)^{2k+1}DQ^{-1/2}x_0\\
=& \ 2(-1)^{j+k+1}\|A\|^{2k+1}(AD)D^{2k+1}Q^{1/2}\xi_0
= 2(-1)^{j+k+1}\|A\|^{2k+1}D^{2k+2}AQ^{1/2}\xi_0\\
= & \ (-1)^{k+1}\|A\|^{2k+2}D^{2k+2}DQ^{-1/2}x_0
\end{align*}
and 
\begin{align*}
& \ (DA)^{2k+2}DQ^{-1/2}x_0=(DA)(DA)^{2k+1}DQ^{-1/2}x_0\\
= & \ 2(-1)^{j+k+1}\|A\|^{2k+1}(DA)D^{2k+1}Q^{1/2}\xi_0
=2(-1)^{j+k+1}\|A\|^{2k+1}D^{2k+2}AQ^{1/2}\xi_0\\
= & \ (-1)^{k+1}\|A\|^{2k+2}D^{2k+2}DQ^{-1/2}x_0.
\end{align*}
If $A \neq 0$, we have therefore proved that for all $m \in \nn$,
\begin{multline}\label{fr13}
\Big( \bigcap_{k=0}^{m}\textrm{Ker}
\big((A^2+\|A\|^2)D^{k}\big)\cap \textrm{Ker}
\big([A,D^k]\big)\Big)\cap \rr^{n}\neq\{0\} \\
\Rightarrow \Big(\bigcap_{k=0}^m\textrm{Ker}\big(\textrm{Re }F_j(\textrm{Im }F_j)^k\big)\Big)\cap \rr^{2n} \neq \{0\}.
\end{multline}
Then, we notice that the Cayley-Hamilton theorem
$$\forall A \in M_d(\rr), \quad \chi_A(A)=0,$$ 
where $\chi_A$ denotes the characteristic polynomial associated to the matrix $A$, implies that 
$$\forall m \in \nn, \quad A^m \in \textrm{Span}_{\rr}(I,A,...,A^{d-1}).$$
It implies in particular that 
\begin{multline}\label{fr12}
\forall m \geq n-1, \quad \Big( \bigcap_{k=0}^{m}\textrm{Ker}
\big((A^2+\|A\|^2)D^{k}\big)\cap \textrm{Ker}
\big([A,D^k]\big)\Big)\cap \rr^{n}\\
=\Big( \bigcap_{k=0}^{n-1}\textrm{Ker}
\big((A^2+\|A\|^2)D^{k}\big)\cap \textrm{Ker}
\big([A,D^k]\big)\Big)\cap \rr^{n}
\end{multline}
and, for all $m \geq 2n-1$, 
\begin{equation}\label{fr11}
\Big(\bigcap_{k=0}^m\textrm{Ker}\big(\textrm{Re }F_j(\textrm{Im }F_j)^k\big)\Big)\cap \rr^{2n}=\Big(\bigcap_{k=0}^{2n-1}\textrm{Ker}\big(\textrm{Re }F_j(\textrm{Im }F_j)^k\big)\Big)\cap \rr^{2n}.
\end{equation}
The proof of Lemma~\ref{lem1.1} is complete when gathering (\ref{fr14}), (\ref{fr13}), (\ref{fr12}) and (\ref{fr11}).

\end{document}